\documentclass[11pt,reqno]{article}

\usepackage{amsfonts}
\usepackage{color}
\usepackage{amsmath}
\usepackage{amsthm}
\usepackage{amssymb}
\usepackage{mathrsfs}
\usepackage{amstext}
\usepackage{graphicx}
\usepackage{tikz}
\numberwithin{equation}{section}

\theoremstyle{plain}
\newtheorem{satz}{Theorem}[section]
\newtheorem{defi}[satz]{Definition}
\newtheorem{cor}[satz]{Corollary}\newtheorem{lem}[satz]{Lemma}

\newtheorem{rem}[satz]{Remark}

\newcommand{\re}{\ensuremath{\mathbb{R}}}\newcommand{\N}{\ensuremath{\mathbb{N}}}
\newcommand{\zz}{\ensuremath{\mathbb{Z}}}
\newcommand{\T}{\ensuremath{\mathbb{T}^d}}\newcommand{\tor}{\ensuremath{\mathbb{
T }}}

\newcommand{\Z}{{\ensuremath{\zz}^d}}
\newcommand{\n}{\ensuremath{{\N}_0}}
\newcommand{\R}{\ensuremath{{\re}^d}}

\newcommand{\cm}{\ensuremath{\mathcal M}}

\newcommand{\cn}{\ensuremath{\mathcal N}}

\newcommand{\rank}{{\rm rank \, }}

\newcommand{\mix}{{\rm mix}}
\newcommand{\eps}{\varepsilon}

\newcommand{\bproof}{\begin{proof}}
\newcommand{\eproof}{\end{proof}}

\newlength{\fixboxwidth}
\newcommand{\be}{\begin{equation}}
\newcommand{\ee}{\end{equation}}
\newcommand{\beq}{\begin{eqnarray}}
\newcommand{\beqq}{\begin{eqnarray*}}
\newcommand{\eeq}{\end{eqnarray}}
\newcommand{\eeqq}{\end{eqnarray*}}

\begin{document}
\title{Approximation of mixed order Sobolev functions on the $d$-torus -- Asymptotics, preasymptotics and $d$-dependence}

\author{Thomas K\"uhn$^a$,  Winfried Sickel$^b$ and Tino Ullrich$^{c,}$\footnote{Corresponding author. Email: tino.ullrich@hcm.uni-bonn.de, Tel: +49 228 73 62224} \\\\
$^a$ University of Leipzig, Augustusplatz 10, 04109 Leipzig, Germany\\
$^b$ Friedrich-Schiller-University Jena, Ernst-Abbe-Platz 2, 07737 Jena, Germany\\
$^c$ Hausdorff-Center for Mathematics, Endenicher Allee 62, 53115 Bonn, Germany}


\date{\today}

\maketitle

\begin{abstract}
We investigate the approximation of $d$-variate periodic functions in Sobolev spaces 
of dominating mixed (fractional) smoothness $s>0$ on the $d$-dimensional torus, where 
the approximation error is measured in the $L_2-$norm.
In other words, we study the approximation numbers of the Sobolev embeddings $H^s_\mix(\tor^d)\hookrightarrow L_2(\tor^d)$, with
particular emphasis on the dependence on the dimension $d$. For any fixed smoothness $s>0$, 
we find the exact asymptotic behavior of the constants as $d\to\infty$. We observe
super-exponential decay of the constants in $d$, if $n$, the number of linear samples of $f$, is large.
In addition, motivated by numerical implementation issues, we also focus on
the error decay that can be achieved by low rank approximations. We present some surprising 
results for the so-called ``preasymptotic'' decay and point out connections to the recently introduced 
notion of quasi-polynomial tractability of approximation problems.

\medskip
\noindent
{\bf Keywords} Approximation numbers $\cdot$ Sobolev spaces of mixed smoothness $\cdot$ rate of convergence $\cdot$ preasymptotics $\cdot$ $d$-dependence 
$\cdot$ quasi-polynomial tractability

\medskip
\noindent
{\bf Mathematics Subject Classifications (2000)} \ 42A10; 41A25; 41A63; 46E35; 65D15

\end{abstract}



\section{Introduction}


In the present paper we investigate the  behavior of the
approximation numbers of the embeddings
$$
    I_d:H^s_\mix(\T) \to L_2(\T)\,,\quad s>0\,,\quad d\in \N\,,
$$
where $H^s_\mix(\T)$ is the periodic Sobolev space  of dominating mixed fractional smoothness $s$
on the $d$-torus $\tor^d$ represented in $\re^d$ by the cube $[0,2\pi]^d$. The approximation numbers of a bounded linear operator $T:X\to Y$ between
two Banach spaces are defined as
\be\label{0002}
 \begin{split}
a_n (T:\, X \to Y)&:= \inf\limits_{\rank A<n} \,\sup\limits_{\|x|X\|\leq 1}\|\, Tx -Ax|Y\|\\
&= \inf\limits_{\rank A<n}\|T-A:X \to Y\| , \qquad n \in \N\,.
 \end{split}
\ee
{~}\\
They describe the best approximation of $T$ by finite rank operators. If $X$ and $Y$ are Hilbert
spaces and $T$ is compact, then $a_n(T)$ is the $n$th singular number of $T$.

The first result on the approximation of Sobolev embeddings is due to
Kolmogorov \cite{Kol}. He showed already in 1936 that in the univariate (homogeneous) case 
with integer
smoothness $m\in \N$ the approximation numbers $a_n(I_d:\dot{H}^m(\tor) \to L_2(\tor))$ 
are given by $n^{-m}$. Here we are interested in the multivariate (inhomogeneous) situation, where
$d$ is large, and investigate the approximation numbers
$a_n(I_d:H^s_{\mix}(\T) \to L_2(\T))$ for arbitrary smoothness parameters $s>0$.
The spaces $H^s_{\mix}(\T)$ are much smaller than the isotropic spaces
$H^s(\T)$, and often they are considered as a reasonable model for reducing the computational effort
in high-dimensional approximation. 

In recent years there has been an
increasing interest in the approximation of multivariate functions, since
many problems, e.g. in finance or quantum chemistry, are modeled in associated
function spaces on high-dimensional domains. 
It has been shown that the functions which have to be approximated often possess a mixed Sobolev regularity, as for instance
eigenfunctions of certain Hamilton operators in quantum chemistry, see Yserentant's lecture note \cite{Ys10} and the references given
there.

In \cite[Theorem~III.4.4]{T93b} the two-sided estimate 
\begin{equation}\label{ws-06}
c_s(d)  \, n^{-s}(\ln n)^{(d-1)s} \le  a_{n} (I_d: \, H^{s}_\mix (\T) \to L_2 (\T)) \le
 C_s(d)  \, n^{-s}(\ln n)^{(d-1)s}\,,
\end{equation}
for $n\in\N$ can be found. Here the constants $c_s(d)$ and $C_s(d)$, depending only on $d$ and $s$, 
were not explicitly determined. 
Some more references and comments to the history of (\ref{ws-06}) will be given in Subsection \ref{lit}.
Our main focus is to clarify, for arbitrary but fixed $s>0$, the
dependence of these constants on $d$. 
In fact, it is necessary to fix the norms on the spaces $H^{s}_\mix (\T)$ in advance, since the 
constants $c_s(d)$ and $C_s(d)$ in \eqref{ws-06} depend on the size of the respective unit balls.
Surprisingly, for a collection of quite natural norms (see the next subsection for details)  it turns out that we can 
choose    
$$
  C_s(d) = \Big[\frac{\lambda^{d}}{(d-1)!}\Big]^s\,,
$$
with $2\leq \lambda \leq 6$ depending on the chosen norm. Note that $C_s(d)$ decays super-exponentially in $d$.
This observation can be compared to similar results in  Bungartz, Griebel \cite{BG04}, 
Griebel \cite{Gr},  Schwab et al.\ \cite{SST08} and D\~ung, Ullrich \cite{DiUl13}, 
where the authors noticed at least exponential decay of the constants. 
A more detailed comparison will be made in Subsection \ref{lit}. 

Let us ignore the constants $c_s(d)$ and $C_s(d)$ for a moment, and fix $s>0$.
Then, for arbitrary $d\in\N$, the function $f_{d}(t):= t^{-s}\, (\ln t)^{s(d-1)}$ is 
increasing on $[1,e^{d-1}]$ and decreasing on $[e^{d-1},\infty)$, hence
its maximum on $[1,\infty)$ is
\[
\max_{t\ge 1} f_{d}(t) = f_{d}(e^{d-1}) = \Big(\frac{d-1}{e}\Big)^{s(d-1)}\, ,
\]
which increases super-exponentially in $d$. 
That means, for large $d$ we have to wait ``exponentially long'' until the 
sequence $n^{-s}(\ln n)^{(d-1)s}$ decays, and even longer until it becomes less than one.
Note that for all norms on $H^{s}_\mix (\T)$ to be considered in this paper, we have 
$a_{n} (I_d: \, H^{s}_\mix (\T) \to L_2 (\T))\le 1$ for all $n$.
Consequently, for small values of $n$ the behavior of 
$a_{n} (I_d: \, H^{s}_\mix (\T) \to L_2 (\T))$ is not properly reflected by the asymptotic rate $n^{-s}(\ln n)^{(d-1)s}$.

This is the reason why we split our investigations into three parts. First, we show that the limit
$$
\lim_{n\to \infty} \, \frac{n^s\cdot a_{n} (I_d: \, H^{s}_\mix (\T) \to L_2 (\T))}{(\ln n)^{(d-1)s}} = \Big[\frac{2^d}{(d-1)!}\Big]^s
$$
exists, having the same value for various norms. Secondly, for exponentially large $n$, we calculate some admissible 
constants $c_s(d)$ and $C_s(d)$. 
Finally, we consider in some detail the situation of small $n$, more precisely in the range $1 \le n\le 4^d$. 
For large $d$ this is the most interesting part for practical issues,
since $4^d$ pieces of information might already be too much for any reasonable algorithm.

The paper is organized as follows. In Section \ref{spaces} we introduce and investigate the Sobolev spaces of interest.
Here we are mainly interested in some assertions on equivalent norms and embeddings.
In Subsection \ref{iso} we add a few remarks to isotropic Sobolev spaces and their relation to 
Sobolev spaces of dominating mixed smoothness. Subsection \ref{parapp}
in this section is devoted to some basics on approximation numbers, in particular, in connection with diagonal operators.
In Section \ref{combin} we study some combinatorial identities and estimates.
Section \ref{number} contains our main results.
The final Section \ref{Erich} transfers our approximation results into the recently introduced notion of quasi-poynomial tractability of the 
respective approximation problems.

{\bf Notation.} As usual, $\N$ denotes the natural numbers, $\N_0$ the non-negative integers,
$\zz$ the integers and
$\re$ the real numbers. With $\tor$ we denote the torus represented by the interval $[0,2\pi]$.
For a real number $a$ we put $a_+ := \max\{a,0\}$ and denote by $[a]$ its greatest integer part.
The letter $d$ is always reserved for the dimension in $\Z$, $\R$, $\N^d$, and $\T$.
For $0<p\leq \infty$ and $x\in \R$ we denote $|x|_p = (\sum_{i=1}^d |x_i|^p)^{1/p}$ with the
usual modification for $p=\infty$. The symbol $\# \Omega$ stands for the cardinality of the set $\Omega$.
If $X$ and $Y$ are two Banach spaces, the norm
of an element $x$ in $X$ will be denoted by $\|x|X\|$ and the norm of an operator
$A:X \to Y$ by $\|A:X\to Y\|$. The symbol $X \hookrightarrow Y$ indicates that there is a
continuous embedding from $X$ into $Y$. The equivalence $a_n\sim b_n$ means that there are constants $0<c_1\le c_2<\infty$ such that 
$c_1a_n\le b_n\le c_2a_n$ for all $n\in\N$.


\section{Preliminaries}\label{spaces}



\subsection{Sobolev spaces of dominating mixed smoothness on the $d$-torus}


All results in this paper are stated for function spaces on the $d$-torus $\T$,
which is represented in the Euclidean space $\R$ by the cube
$\tor^d = [0,2\pi]^d$, where opposite sides are identified.
In particular, for functions $f$ on $\tor$, we have $f(x) = f(y)$ whenever $x-y = 2\pi k$ for some $k\in \Z$. These
functions can be viewed as $2\pi$-periodic in each component.

The space $L_2(\T)$ consists of all (equivalence classes of) measurable functions $f$ on $\T$
such that norm
$$
    \|f|L_2(\T)\|:=\Big(\int_{\T} |f(x)|^2\,dx\Big)^{1/2}
$$
is finite. All information on a function $f\in L_2(\T)$ is encoded
in the sequence $(c_k(f))_k$ of its Fourier coefficients, given by
\[
c_k (f):= \frac{1}{(2\pi)^{d/2}} \, \int_{\T} \, f(x)\, e^{-ikx}\, dx\, , \qquad k
\in
\Z\, .
\]
Indeed, we have Parseval's identity
\be\label{Pars}
    \|f|L_2(\T)\|^2 = \sum\limits_{k\in \Z} |c_k(f)|^2
\ee
as well as
\[
f(x) = \frac{1}{(2\pi)^{d/2}}\, \sum_{k \in \Z} \, c_k (f)\, e^{ikx}
\]
with convergence in $L_2(\T)$.

The mixed Sobolev space $H^m_\mix(\T)$ of integer smoothness $m\in \N$ is the collection
of all $f\in L_2 (\T)$ such that all
distributional derivatives $D^\alpha  f$ of order $\alpha = (\alpha_1,...,\alpha_d)$ with
$\alpha_j\le m$, $j=1,...,d$, belong to $L_2 (\T)$.
We put
\begin{equation}\label{t-1}
\| \, f \, |H^m_{\mix} (\T)\| := \Big(
\sum_{|\alpha|_\infty\le m}
\| \, D^\alpha  f \, |L_2 (\T)\|^2\Big)^{1/2} \, .
\end{equation}
One can rewrite this definition in terms of Fourier coefficients.
Taking $c_k(D^{\alpha}f) = (ik)^{\alpha}c_k(f)$ into account, Parseval's identity \eqref{Pars} 
implies (using the convention $0^0 = 1$)
\beq\label{norm2}
\| \, f \, |H^m_{\mix} (\T)\|^2 & = &
\sum_{|\alpha|_\infty\le m}
\Big\| \, \frac{1}{(2\pi)^{d/2}}\sum_{k \in \Z} \, c_k (f)\, (ik)^\alpha  e^{ikx} \,
\Big|L_2 (\T)\Big\|^2
\label{natural}
\\\nonumber
& = & \sum\limits_{k\in \Z} |c_k(f)|^2 \, \Big(\sum_{|\alpha|_\infty\le m}
\nonumber
\prod\limits_{j=1}^d |k_j|^{2\alpha_j}\Big)
= \sum\limits_{k\in \Z} |c_k(f)|^2 \, \Big(\prod\limits_{j=1}^d\sum\limits_{\alpha = 0}^{m}
|k_j|^{2\alpha}\Big)\\
\nonumber
&=& \sum\limits_{k\in \Z} |c_k(f)|^2 \, \Big(\prod\limits_{j=1}^d\sum\limits_{\alpha = 0}^{m}
|k_j|^{2\alpha}\Big) =  \sum\limits_{k\in \Z} |c_k(f)|^2 \,\prod\limits_{j=1}^d v_m(k_j)^2\,,
\nonumber
\eeq
where
\be\label{f91}
  v_m(\ell)^2 = \sum\limits_{\alpha = 0}^m |\ell|^{2\alpha} 
\ee
Due to our convention $0^0=1$ we have $v_m(0)=1$, moreover $v_m(\pm 1)=m+1$.
Defining 
\be\label{wm(k)}
w_m(k) := \prod_{j=1}^d v_m(k_j)\qquad\text{for } k = (k_1,...,k_d) \in \Z\,, 
\ee
we obtain
$$
   \| \, f \, |H^m_{\mix} (\T)\| = \Big[\sum\limits_{k\in \Z} |c_k(f)|^2 w_m(k)^2\Big]^{1/2}\,.
$$
We could have also started with the equivalent norm
\be\label{norm3}
\| \, f \, |H^m_\mix (\T)\|^* := \Big(\sum_{\alpha \in \{0,m\}^d} \| \, D^{\alpha}f\, |L_2
(\T)\|^2\Big)^{1/2} \, .
\ee
%
Similarly as above, a reformulation of (\ref{norm3}) in terms of Fourier coefficients yields 
\be\label{norm4}
\| \, f \, |H^m_\mix (\T)\|^*  =
 \,
 \Big[\sum_{k \in \Z} \, |c_k (f)|^2  \prod\limits_{j=1}^d (1+|k_j|^{2m}) \Big]^{1/2}\,.
\ee

Inspired by  \eqref{natural} and \eqref{norm4} we define Sobolev spaces of dominating 
mixed smoothness of fractional order $s>0$ as follows.

\begin{defi}\label{varnorms}
Let $s>0$. The periodic Sobolev space $H^s_{\mix}(\T)$ is the collection of all $f\in L_2(\T)$
such that $\|f|H^s_{\mix}(\T)\|<\infty$, where\\
{\em (i)} the classical (natural) norm $\|f|H^s_{\mix}(\T)\|^+$ is defined as
\be\label{norm4.5}
   \| \, f \, |H^s_{\mix} (\T)\|^+  :=
\, \Big[
 \sum_{k \in \Z} \, |c_k (f)|^2 \prod\limits_{j=1}^d\big(1+|k_j|^2\big)^{s} \Big]^{1/2}\,,
\ee
{\em (ii)} the modified classical norm $\|f|H^s_{\mix}(\T)\|^{*}$ is defined as
\be\label{norm5}
\| \, f \, |H^s_{\mix} (\T)\|^*  :=
\, \Big[
 \sum_{k \in \Z} \, |c_k (f)|^2 \prod\limits_{j=1}^d\big(1+ |k_j|^{2\,s}\big)\Big]^{1/2}\,,
\ee
\\
{\em (iii)} and the norm $\|f|H^s_{\mix}(\T)\|^{\#}$ is a further modification defined by
$$
\| \, f \, |H^s_{\mix} (\T)\|^\#  :=
\, \Big[
 \sum_{k \in \Z} \, |c_k (f)|^2 \prod\limits_{j=1}^d\big(1+ |k_j|\big)^{2s} \Big]^{1/2}\,.
$$
\end{defi}

In the sequel we will often use the notation $H^{s,+}_{\mix}(\T)$, $H^{s,*}_{\mix}(\T)$, 
and $H^{s,\#}_{\mix}(\T)$ to indicate which of these equivalent 
norms on $H^s_{\mix}(\T)$ we are considering. For integer smoothness $s=m \in \N$ all three norms 
are also equivalent to the norm given in (\ref{t-1}).
Moreover, in some special cases we do not only have equivalence, but even equality of the norms,
namely 
$$
\|\cdot|H^1_{\mix}(\T)\| = \|\cdot|H^1_{\mix}(\T)\|^{*} = \|\cdot|H^1_{\mix}(\T)\|^+ 
$$
and 
$$
\|\cdot|H^{1/2}_{\mix}(\T)\|^{*} = \|\cdot|H^{1/2}_{\mix}(\T)\|^{\#}\, .
$$

Clearly, the size of the unit balls with respect to equivalent norms can be significantly different.
Or, in other words, switching from one to another equivalent norm might produce equivalence constants which
badly depend on the dimension $d$. Since we are interested in situations where $d$ is large or even $d\to \infty$, we 
have to be very careful with these equivalence constants.
Therefore, in this context, norm one embeddings are of particular interest and will be very useful.
The embeddings given in the next lemma are due the monotonicity
of the norms $|\cdot|_p$\,, where $0<p<\infty$, except (v), which is a 
simple consequence of the fact that the square of an integer is larger than its absolute value.

\begin{lem}\label{norm_one} Let $s>0$. The following embeddings have norm one.\\
{\rm(i)} If $s \geq 1$, then
$$
    H^{s,\#}_{\mix}(\T) \hookrightarrow H^{s,+}_{\mix}(\T) \hookrightarrow H^{s,*}_{\mix}(\T)\,,
$$
{\rm(ii)} if $1/2 \leq s \leq 1$, then
$$
    H^{s,\#}_{\mix}(\T) \hookrightarrow H^{s,*}_{\mix}(\T) \hookrightarrow H^{s,+}_{\mix}(\T) \,,
$$
{\rm(iii)} if $s\leq 1/2$, then
$$
   H^{s,*}_{\mix}(\T) \hookrightarrow H^{s,\#}_{\mix}(\T) \hookrightarrow H^{s,+}_{\mix}(\T)\,,
$$
{\rm(iv)} if $s>t$, then
$$
  H^{s,+}_{\mix}(\T) \hookrightarrow H^{t,+}_{\mix}(\T)\quad,\quad
  H^{s,*}_{\mix}(\T) \hookrightarrow H^{t,*}_{\mix}(\T)\quad,\quad
  H^{s,\#}_{\mix}(\T) \hookrightarrow H^{t,\#}_{\mix}(\T)\,,
$$
{\rm(v)} and finally,
$$H^{s,+}_{\mix}(\T) \hookrightarrow H^{s/2,\#}_{\mix}(\T)\,.$$
\end{lem}

\noindent We also have embeddings where $H^m_{\mix}(\T)$ is involved. 

\begin{lem}\label{norm_two} Let $m \in \N$. 
Then for all $f \in H^m_{\mix}(\T)$ the following chain of inequalities holds.
\begin{equation}\label{ws-03}
\|\, f\, |H^m_{\mix}(\T)\|^* \le 
\|\, f\, |H^m_{\mix}(\T)\| \le  \|\, f \, |H^m_{\mix}(\T)\|^+ \le \left(\frac{2^m}{m+1}\right)^{d/2}
\|\, f\, |H^m_{\mix}(\T)\|
\end{equation}
\end{lem}
\bproof The first inequality in (\ref{ws-03}) is obvious. The second one is a consequence of 
\eqref{natural} and \eqref{f91}, together with 
the fact that $v_{m}(\ell)^2 \leq (1+|\ell|^2)^{m}$ for all $\ell\in \zz$ and $m\in \N$. For the third
inequality, it is enough to notice that
$$
    \|I_d:H^m_{\mix}(\T) \to H^{m,+}_{\mix}(\T)\|^2 
    = \sup\limits_{k\in \Z} \prod_{j=1}^d\frac{(1+|k_j|^2)^{m}}{v_m(k_j)}
    =\left(\sup\limits_{k\in \N} \frac{(1+k^2)^{m}}{1+k^2+...+k^{2m}}\right)^d\,
$$
and that the function
$f(x)=\frac{(1+x)^m}{1+x+...+x^m}$ is decreasing on $[1,\infty)$, hence $f(x)\ge f(1)=\frac{2^m}{m+1}$.
%
\eproof

The most convenient norm for our purposes is 
$\|\, \cdot \, |H^s_{\mix}(\T)\|^\#$.
In almost all combinatorial estimates given below we use this specific norm.
Afterwards, with some additional effort, the results are carried over to the less convenient but more important norms 
$\|\, \cdot \, |H^{s}_{\mix}(\T)\|^+$, $\|\, \cdot \, |H^s_{\mix}(\T)\|^*$ and  $\|\, \cdot \, |H^m_{\mix}(\T)\|$.


\subsection{Isotropic Sobolev spaces  on the $d$-torus}
\label{iso}

Let $m \in \N$.
Then the isotropic Sobolev space $H^ m (\T)$ is the collection of all $f\in L_2 (\T)$ such that all 
distributional derivatives $D^\alpha  f$ up to order $m$ belong to $L_2 (\T)$, i.e.,
$$
\| \, f \, |H^m (\T)\| := \Big(
\sum_{|\alpha|_1\le m} \| \, D^\alpha  f \, |L_2 (\T)\|^2\Big)^{1/2} < \infty \, .
$$
Fractional versions for $s>0$ can be introduced by using Fourier coefficients 
and the norm
$$
\| \, f \, |H^s (\T)\|^+  := 
 \, 
 \Big(\sum_{k \in \Z} \, |c_k (f)|^2 (1+ \sum_{j=1}^d  
|k_j|^{2})^s \Big)^{1/2} \, .
$$ 
Based on these norms it is easy to compare the isotropic Sobolev spaces
with the Sobolev spaces of dominating mixed smoothness.

\begin{lem}
 \label{embe}
 Let $s>0$. Then we have the chain of continuous embeddings
 \begin{equation}\label{ws-05}
 H^{sd} (\T) \hookrightarrow H^s_{\mix} (\T) \hookrightarrow H^s (\T)\,,
 \end{equation}
 and this is best possible, i.e., for all $\varepsilon >0$, 
 
 \[
  H^{sd -\varepsilon} (\T) \not \subset  H^s_{\mix} (\T)\not \subset  H^{s+\varepsilon} (\T)\,.
 \]

 \end{lem}

\bproof The proof is elementary, so we will omit the details.
However, it is of certain interest to note that the embedding operators in (\ref{ws-05}) are always of norm one, i.e. 
\[
\| I_d:H^{sd,+} (\T) \to  H^{s,+}_{\mix} (\T)\| =  \| I_d:H^{s,+}_{\mix} (\T) \to  H^{s,+} (\T)\| = 1
\]
for all $s>0$, and 
\[
\| I_d:H^{md} (\T) \to  H^{m}_{\mix} (\T)\| =  \| I_d:H^{m}_{\mix} (\T) \to  H^{m} (\T)\| = 1
\]
for all $m \in \N$.
\eproof

The mixed space $H^s_{\mix} (\T)$ is much closer to the space on the left-hand side in (\ref{ws-05})
than to the space on the right-hand side.
This is indicated by a short look at the behavior of the approximation numbers.
It is known, see, e.g., \cite[Chapter~2, Theorems~4.1 and 4.2]{T93b}, that
\begin{equation}\label{ws-07}
a_s(d)  \, n^{-s/d} \le  a_{n} (I_d: \, H^{s} (\T) \to L_2 (\T)) \le
 A_s(d)  \, n^{-s/d} \, , \qquad  n \in \N\, ,
\end{equation}
holds for all $n$ with constants $a_s(d)$ and $A_s(d)$, only depending on $d$ and $s$\,,
and hence 
\[ 
a_{n} (I_d: \, H^{sd} (\T) \to L_2 (\T)) \sim n^{-s}\, . 
\]
This coincides up to a logarithmic perturbation with the behavior of 
$a_{n} (I_d: \, H^{s}_{\mix} (\T) \to L_2 (\T))$, see  (\ref{ws-06}).
Roughly speaking, the mixed Sobolev spaces
$ H^{s}_{\mix} (\T) $ are much smaller than their isotropic counterparts  $H^{s} (\T)$.
The behavior of the associated approximation numbers is almost the same as in the 
one-dimensional isotropic case $H^s(\tor)$.
>From the very  beginning this has been a major motivation to consider spaces of dominating mixed smoothness 
in approximation theory as well as in the field of information based complexity (IBC).
We refer to Babenko \cite{Ba}, Mityagin \cite{Mi} and Smolyak \cite{Sm} for early contributions 
in the framework of approximation theory (these references are also of relevance with respect to (\ref{ws-06})).
More recent results may be found in Temlyakov's monograph \cite{T93b}.
The role of the spaces $H^s_{\mix} (\T)$ in IBC
is summarized in the recent series of books by Novak and Wo{\'z}niakowski
\cite{NoWo08,NoWo10,NoWo12}. Observe that in IBC the spaces are
sometimes called Korobov spaces, see, e.g., \cite[pp.~341]{NoWo08}.

\begin{rem}
 \rm
In \cite{KSU1} we gave a proof of (\ref{ws-07}) with explicit constants 
$a_s(d)$ and $A_s(d)$ for various equivalent norms.
\end{rem}


\subsection{Approximation numbers}
\label{parapp}


If $\tau = (\tau_n)_{n=1}^{\infty}$ is a sequence of real numbers with $\tau_1 \geq \tau_2 \geq ... \geq 0$\,,
we define the diagonal operator
$  D_{\tau}:\ell_2 \to \ell_2$ by $D_{\tau}(\xi) = (\tau_n \xi_n)_{n=1}^{\infty}$.
Recall the definition of the approximation numbers \eqref{0002} already given in the introduction.
The following fact concerning approximation numbers of diagonal operators is well-known,
see e.g. K\"onig \cite[Section 1.b]{Ko}, Pinkus \cite[Theorem IV.2.2]{Pin}, and  Novak and Wo{\'z}niakowski
\cite[Corollary 4.12]{NoWo08}. Comments on the history may be found in
Pietsch \cite[6.2.1.3]{Pi07}.
\begin{lem}\label{sing}
Let $\tau$ and $D_\tau$ be as above.
Then
\[
a_n (D_\tau:\, \ell_2 \to \ell_2) = \tau_n\, , \qquad n \in \N\, .
\]
\end{lem}
\noindent Here the index set of $\ell_2$ is $\N$. We need a modification for arbitrary 
countable index sets $J$. Then the space $\ell_2(J)$ is the collection of all
$\xi = (\xi_j)_{j\in J}$ such that the norm
$$
    \|\xi|\ell_2(J)\|:=\Big(\sum\limits_{j\in J}|\xi_j|^2\Big)^{1/2}
$$
is finite. Let $w = (w_j)_{j\in J}$ with $w_j>0$ for all $j\in J$, and assume that for every $\delta>0$ there are
only finitely many $j\in J$ with $w_j\geq \delta$\,. Then the non-increasing rearrangement
$(\tau_n)_{n\in \N}$ of $(w_j)_{j\in J}$ exists, and $\lim_{n\to\infty} \tau_n = 0$.
Defining $D_w:\ell_2(J) \to \ell_2(J)$ by $D_w(\xi) = (w_j\xi_j)_{j\in J}$ for $\xi \in \ell_2(J)$, Lemma\
\ref{sing} gives
$$
  a_n(D_w:\ell_2(J)\to \ell_2(J)) = \tau_n\,.
$$
The preceding identity is scalable in the following sense.

\begin{lem}\label{scale} Let $J$ be a countable index set, let $w=(w_j)_{j\in J}$ and $(\tau_n)_{n\in \N}$ be as above. Then,
setting $w^s=(w_j^s)_{j\in J}$, one has for any $s>0$
$$
  a_n(D_{w^s}:\ell_2(J)\to \ell_2(J)) = a_n(D_{w}:\ell_2(J)\to \ell_2(J))^s = \tau^s_n\,.
$$

\end{lem}

Now we can reduce our problem on embedding operators in function spaces to the considerably simpler context of diagonal operators in sequence spaces, where index set is $J=\Z$\,. To this end, we consider the operators
$$
  A_s:H^{s,+}_{\mix}(\T) \to \ell_2(\Z)\qquad\mbox{and}\qquad B_s:\ell_2(\Z) \to H^{s,+}_{\mix}(\T)
$$
defined by
$$
   A_sf = (w_{s}^+(k)c_k(f))_{k\in \Z}\qquad\mbox{and}\qquad B_s\xi = (2\pi)^{-d/2}\sum\limits_{k\in \Z}
   \frac{\xi_k}{w_{s}^+(k)} e^{ikx}\,,
$$
where the weights are $w_{s}^+(k) := \prod_{j=1}^d\big(1+|k_j|^2\big)^{s/2}\,$.
Note the semigroup property of these weights, i.e., $w_{s}^+(k)\cdot w_{t}^+(k) = w_{s+t}^+(k)$.
Furthermore, we put for $k\in \Z$ 
$$w(k) := \frac{w_{s_1}^+(k)}{w_{s_0}^+(k)}$$
and make use of the associated diagonal operator $D_w$.
Then the following commutative diagram illustrates the situation
quite well in case $s_0>s_1 \geq 0$:

\tikzset{node distance=3cm, auto}

\begin{center}
\begin{tikzpicture}
  \node (H) {$H^{s_0,+}_{\mix}(\T)$};
  \node (L) [right of =H] {$H^{s_1,+}_{\mix}(\T)$};
  \node (ell) [below of=H] {$\ell_2(\Z)$};
  \node (ell2) [right of=ell] {$\ell_2(\Z)$};
  \draw[->] (H) to node {$I_d$} (L);
  \draw[->] (H) to node {$A_{s_0}$} (ell);
  \draw[->] (ell) to node {$D_w$} (ell2);
  \draw[->] (ell2) to node {$B_{s_1}$} (L);
\end{tikzpicture}
\end{center}

By the definition of the norm $\|\cdot|H^s_{\mix}(\T)\|^+$ it is clear
that $A_s$ and $B_s$ are isometries and $B_s = A_s^{-1}$. For the embedding
$I_d:H^{s_0,+}_{\mix}(\T) \to H^{s_1,+}_{\mix}(\T)$ if $s_0>s_1 \geq 0$ we obtain the factorization
\be\label{eq071}
    I_d =  B_{s_1} \circ D_w \circ A_{s_0}\,.
\ee
The multiplicativity of the approximation numbers applied to \eqref{eq071} implies
$$
  a_n(I_d) \leq \|A_{s_0}\|a_n(D_w)\|B_{s_1}\| = a_n(D_w) = \tau_n\,,
$$
where $(\tau_n)_{n=1}^{\infty}$ is the non-increasing rearrangement of
$(w(k))_{k\in \Z}$\,. The reverse inequality can be shown analogously. This gives the
important identity
\be\label{eq072}
    a_n(I_d) = a_n(D_w) = \tau_n\,.
\ee
Of course, \eqref{eq072} also holds for $I_d:H^{s_0,\#}_{\mix}(\T) \to H^{s_1,\#}_{\mix}(\T)$ and for
$I_d:H^{s_0,*}_{\mix}(\T) \to H^{s_1,*}_{\mix}(\T)$ with the obvious
adaption of the weights. Due to the semigroup property mentioned above and Lemma \ref{scale}
we have in particular the nice properties
\be\label{eq073}
  \begin{split}
    a_n(I_d:H^{s_0,+}_{\mix}(\T)\to H^{s_1,+}_{\mix}(\T)) &= a_n(I_d:H^{s_0-s_1,+}_{\mix}(\T)\to L_2(\T))\\
    &= a_n(I_d:H^{1,+}_{\mix}(\T)_{\mix}\to L_2(\T))^{s_0-s_1}\,
  \end{split}
 \ee
and
\be\nonumber
  \begin{split}
    a_n(I_d:H^{s_0,\#}_{\mix}(\T)\to H^{s_1,\#}_{\mix}(\T)) &= a_n(I_d:H^{s_0-s_1,\#}_{\mix}(\T)\to L_2(\T))\\
    &= a_n(I_d:H^{1,\#}_{\mix}(\T)\to L_2(\T))^{s_0-s_1}\,.
  \end{split}
\ee
For the norm $\|\cdot\|^{*}$ the corresponding weights are
$$
  w_{s}^*(k) = \prod\limits_{j=1}^d\big(1+|k_j|^{2s}\big)^{1/2}\,.
$$
Note, that they do not satisfy the semigroup property, whence a counterpart of \eqref{eq073} does not hold.


\section{Some combinatorics}
\label{combin}


In most considerations below, a crucial role will be played by
the cardinality $C(r,d)$ of the set
\[
 {\mathcal N}(r,d):=\Big\{ k \in \Z: \quad \prod\limits_{i=1}^d (1+|k_j|) \le r\Big\}\, , \qquad
r \in \N\, .
\]

\begin{lem}\label{ook}
For $r\in \N$ we have
\be\label{eq21}
  C(r,d) = 1 + \sum\limits_{\ell = 1}^{\min\{d,\log_2 r\}} 2^{\ell}\binom{d}{\ell}A(r,\ell)\,,
\ee
where $A(r,\ell) := \#\mathcal{M}(r,\ell)$ with
$$
    \mathcal{M}(r,\ell) = \Big\{k\in \N^\ell:\prod\limits_{j=1}^\ell (1+k_j) \leq r\}\,.
$$
\end{lem}

\bproof
The proof is straightforward. The first summand $1$ in \eqref{eq21} represents the case 
$k_1= \ldots =k_d =0$.
Next we group together those vectors $k$ having exactly $\ell$ non-zero components.
This explains why the summation is running from $1$ to $\min\{d, \log_2 r\}$.
Of course, we may concentrate on those $k \in \Z$ with nonnegative components. 
Since we have $\ell$ non-zero components, this leads to the factor $2^\ell$.
Finally, the binomial coefficient $\binom{d}{\ell}$ is just the number of subsets of $\{1, \ldots,d\}$ of cardinality $\ell$.
\eproof

Later on we need estimates of the quantities $A(r,d)$ for all $r\in\N$.
Obviously we have $A(r,d)=0$ for $1\le r < 2^d$, and  $A(2^d,d)=1$.
We intend to relate the number $A(r,\ell)$ to the $\ell$-dimensional Lebesgue measure of the set
$$
    \mathcal{H}^{\ell}_r:=\Big\{x\in \re^{\ell}~:~x_j\geq 1, j=1,...,\ell, \prod\limits_{j=1}^\ell x_j \leq r\}
    \subset \re^{\ell}\,.
$$
Here, arbitrary real numbers $r>1$ are admitted.

\begin{rem}
 \rm
Of course, $\mathcal{H}^{\ell}_r$ is essentially the restriction of the hyperbolic cross with parameter $r$ to the first octant 
in $\re^\ell$. Knowing the classical approximation strategies with respect 
to the function spaces $H^s_{\mix} (\T)$\,, it is not a surprise that hyperbolic crosses show up here.
For easier reference we concentrate on dyadic hyperbolic crosses.
For $m,d \in \N$ let 
\[
H(m,d) :=
\Big\{k \in \Z:\: \exists u_1,...,u_d \in \N_0 \quad \mbox{s.t.}\quad
      |k_j| \leq 2^{u_j}\: \text{and}\:  \sum_{j=1}^d u_j = m\Big\}\,.
\]
Denote by 
\[
S_m f (x):= (2\pi)^{-d/2}\, \sum_{k \in H(m,d)} c_k (f)\, e^{ikx}\, , \qquad m \in \N\, , 
\]
the associated sequence of partial sums of the Fourier series.
Then $N(m) = \rank S_m \sim m^{d-1}\, 2^m$ and 
\begin{equation}\nonumber
\|I_d -S_m:H^{s}_{\mix} (\T) \to L_2 (\T)\| \sim a_{N(m)} (I_d: H^s_{\mix} (\T)\to L_2 (\T)) \sim N(m)^{-s}(\ln N(m))^{(d-1)s}\,.
\end{equation}
Here all constants behind $\sim$ are independent of $m \in \N$, but depending on $s$ and $d$, see
Bugrov \cite{Bu}, Nikol'skaya \cite{Nia}, Temlyakov \cite{T93b} and \cite{SU1,SU}.
\end{rem}

Let us return to \eqref{eq21}. Our next goal will be two-sided estimates for $A(r,\ell)$. Define the function $v_{\ell}(r):=\mbox{vol}_\ell(\mathcal{H}^{\ell}_r)$.

\begin{lem}\label{vol} Let $\ell,r \in \N$ and $r\geq 2^{\ell}$. Then we have\\
{\em (i)}
\be\label{eq23}
  2^{\ell}v_\ell(r/2^{\ell})\leq A(r,\ell) \leq v_\ell(r)
\ee
and {\em (ii)}
$$
    r\Big(\frac{(\ln r)^{\ell-1}}{(\ell-1)!}-\frac{(\ln r)^{\ell-2}}{(\ell-2)!}\Big)\leq v_{\ell}(r) 
\leq r\frac{(\ln r)^{\ell-1}}{(\ell-1)!}\,,\quad \ell = 2,3,....
$$
Moreover, the upper estimate in (ii) holds as well in case $\ell =1$.
\end{lem}
\bproof For $k \in \N^{\ell}$ put $Q_k := k+[0,1]^{\ell}$. Then it holds
\be\label{eq22}
	      \Big\{x\in \re^{\ell}~:~x_j\geq 2, \prod\limits_{j=1}^\ell x_j \leq r\}
    \subset \bigcup\limits_{k\in \mathcal{M}(r,\ell)}Q_k \subset \Big\{x\in \re^{\ell}~:~x_j\geq 1, \prod\limits_{j=1}^\ell x_j \leq r\}\,.
\ee
Taking $\mbox{vol}_{\ell}$ in \eqref{eq22} together with a change of variable gives (i).

Let us prove (ii) by induction on $\ell$. We first define the function
$$
f_{\ell}(r) := r\frac{(\ln r)^{\ell-1}}{(\ell-1)!}
$$ 
and rewrite (ii) as
\be\label{eq25}
  f_{\ell}(r)-f_{\ell-1}(r) \leq v_{\ell}(r) \leq f_{\ell}(r)\,.
\ee
We consider the upper bound first. One easily verifies the right-hand side in \eqref{eq25} in case $\ell = 1$.
For $\ell\geq $ we use the recurrence relation
$$
    v_{\ell+1}(r) = \int_{1}^r v_{\ell}(r/t)\,dt\,,
$$
which is a consequence of Fubini's theorem. By a change of variable this can be rewritten as
\be\label{eq24}
  v_{\ell+1}(r) = r\int_{1}^r v_{\ell}(s)\,\frac{ds}{s^2}\,.
\ee
This implies
$$
  v_{\ell+1}(r) = r\int_{1}^r v_{\ell}(s)\,\frac{ds}{s^2} \leq r\int_{1}^r f_{\ell}(s)\frac{ds}{s^2} = f_{\ell+1}(r)\,.
$$
Indeed, the substitution $u=\ln s$ yields
\be\label{eq26}
  r\int_{1}^r f_{\ell}(s)\frac{ds}{s^2}=r\int_{1}^r \frac{s(\ln s)^{\ell-1}}{(\ell-1)!}\frac{ds}{s^2} = \frac{r}{(\ell-1)!}\int_{0}^{\ln r} u^{\ell-1}\,du = \frac{(\ln r)^{\ell}}{\ell!}=f_{\ell+1}(r)\,.
\ee
For the lower bound we first verify the left-hand side in \eqref{eq25} in case $\ell = 2$ by using  $v_2 (r)= r \ln r -r + 1$. 
The recurrence
relation \eqref{eq24} together with the induction hypothesis yields
$$
    v_{\ell+1}(r) = r\int_{1}^r v_{\ell}(s)\frac{ds}{s^2} \geq
    r\int_{1}^r \Big(f_{\ell}(s) - f_{\ell-1}(s)\Big)\, \frac{ds}{s^2} = f_{\ell+1}(r) - f_{\ell}(r)\,,
$$
where the last identity is a consequence of \eqref{eq26}\,. The proof is complete.
\eproof

\begin{rem} \rm In the recent preprint \cite{DiCh13}, D\~ung and Chernov
considered cardinalities and volumes of hyperbolic cross type 
sets in $\re^d$ similar to $\mathcal{H}^{\ell}_r$ above, see for instance (1.9), (1.10), Theorem\ 4.2, and Corollaries 4.3., 4.4, 4.5. However, for our purpose, i.e. the control of the numbers $C(r,d)$, see \eqref{eq21} above, the estimates presented here are more appropriate. 
\end{rem}


\section{Approximation numbers of Sobolev embeddings}
\label{number}


In this section we will compute, or at least estimate, the approximation numbers of the embedding 
$I_d: \, H^s_{\mix} (\T) \to L_2 (\T)$\,.
The main aim is to prove (\ref{ws-06}) with explicit constants 
$c_s(d)$ and $C_s (d)$.
First we deal with the norm $\|\, \cdot \, |H^s_\mix(\T)\|^\#$.


\subsection{The approximation numbers $a_n(I_d:\, H^{s,\#}_{\mix}(\T)\to L_2 (\T))$ for large $n$}
\label{numberlarge}

For $s>0$ we put
\[
 w_{s}^\#(k):= \prod\limits_{j=1}^d(1+ |k_j|)^{s}\, , \qquad k \in \Z \, .
\]
Due to Lemma \ref{sing} and \eqref{eq072} we have
\[
a_{n} (I_d: \, H^{s,\#}_{\mix} (\T) \to L_2 (\T)) = \sigma_n\, , \qquad n \in \N\, ,
\]
where $(\sigma_n)_{n\in\N}$ denotes the non-increasing rearrangement of $(1/w_{s}^\#(k))_{k \in \Z}$. We have
$$
\{\sigma_n: n\in\N\}=\{1/w_s^\#(k): k\in\Z\}=\{r^{-s}: r\in\N_0\}\,,
$$
that means $(\sigma_n)_{n\in\N}$ is a piecewise constant sequence. Recall the notation
\[
C(r,d)=\#\Big\{ k \in \Z: \quad \prod\limits_{i=1}^d (1+|k_j|) \le r\Big\}
\,=\,
\#\Big\{k \in \Z: \: w_{s}^\#(k)\le r^s \Big\}\, .
\]
These observations imply the following result.

\begin{lem}\label{an1}
Let $s>0$ and $r \in \N$. Then 
$$
a_{n} (I_d: \, H^{s,\#}_{\mix} (\T) \to L_2 (\T)) = r^{-s}\,, \qquad \mbox{if}
\quad C(r-1,d) < n \le  C(r,d)\, .
$$
\end{lem}

\begin{rem}
\rm Of course, without precise information on the behavior of the quantities $C(r,d)$, Lemma \ref{an1} is not very useful 
for practical purposes. But it provides, at least in 
principle, complete knowledge on the sequence of approximation numbers $a_n (I_d: \, H^{s,\#}_{\mix} (\T) \to L_2 (\T))$.
In particular, 
\begin{align*}\label{ws-08}
a_1 = 1 > &\, a_2 = \ldots = a_{2d+1} = 2^{-s} \\
> & \,a_{2d+2} = \ldots =a_{4d+1}=3^{-s}\\
 > & \,a_{4d+2}=\ldots = a_{2d^2+4d+1}=4^{-s}>\ldots \, .
\end{align*}

Furthermore, for any $n\in \N$, we can easily construct optimal algorithms $S_n$ of rank less than $n$.
If $C(r-1,d)< n \le  C(r,d)$,  we choose 
\[
S_n f (x):= \sum_{k \in \mathcal{N}(r-1,d)} c_k (f) \, e^{ikx}\, . 
\]
In fact, by this construction we get 
$$
    \sup\limits_{\|f|H^s_{\mix}(\T)\|^{\#}\leq 1}\|f-S_nf|L_2(\T)\| = r^{-s} = a_{n}(I_d: \, H^{s,\#}_{\mix} (\T) \to L_2 (\T))\,.
$$
\end{rem}

In a next step we determine the asymptotic behavior of the approximation numbers as $n\to \infty$, including the exact 
dependence on the smoothness parameter $s$ and the dimension $d$. Note that the {\it existence} of the limit in the 
following result is not at all obvious a priori.

\begin{satz}\label{lim1}
 Let $s>0$ and $d\in \N$. Then
 $$
 \lim_{n \to \infty} \, \frac{n^s\,a_{n} (I_d: \, H^{s,\#}_\mix (\T) \to L_2 (\T))}{(\ln n)^{(d-1)s} \, } =
 \Big[\frac{2^{d}}{(d-1)!}\Big]^s \, .
 $$
\end{satz}

\bproof  Fix $d\in\N$. By Lemma \ref{scale} it is enough to deal with the case $s=1$. 
For simplicity of notation we write 
$a_n:=a_{n} (I_d: \, H^{1,\#}_\mix (\T) \to L_2 (\T))$. We
have $a_n = 1/r$ if $r\in\N$ and $C(r-1,d) < n \leq C(r,d)$, see Lemma \ref{an1}. 
Clearly $\lim_{r\to\infty}C(r,d)=\infty$, moreover
the sequence $n (\ln n)^{-(d-1)}$ is increasing
for $n>e^{d-1}$. 
Hence we obtain for sufficiently large $r\in\N$ the two-sided inequality

\be\label{eq28}
\frac{C(r-1,d)}{r(\ln C(r-1,d))^{d-1}}\le 
\frac{n\,a_n}{(\ln n)^{d-1}} \le   \frac{C(r,d)}{r(\ln C(r,d))^{d-1}}\,.
\ee
By \eqref{eq21} and \eqref{eq23} we have for $r\ge 2^d$
\be
    C(r,d)\leq 1+\sum\limits_{\ell = 1}^d\binom{d}{\ell}2^\ell v_\ell(r)\, .
\ee
>From $C(r,d) \geq C(r,1)=2r-1 \geq r$ for all $r\in\N$, we get $\ln C(r,d)\ge \ln r$. 
Taking Lemma \ref{vol}/(ii) into account, we arrive at
$$
   \frac{C(r,d)}{r(\ln C(r,d))^{d-1}}\leq \frac{1+\sum\limits_{\ell = 1}^d\binom{d}{\ell}2^\ell r
\frac{(\ln r)^{\ell-1}}{(\ell-1)!}}{r(\ln r)^{d-1}} \xrightarrow[r\to\infty]{} \frac{2^d}{(d-1)!}\,,
$$
since only the last summand contributes to the limit.
Together with \eqref{eq28} this gives
$$
    \limsup\limits_{n\to\infty}\frac{n\,a_n}{(\ln n)^{d-1}} \leq \frac{2^d}{(d-1)!}\,.
$$
Now let us pass to the estimate from below. By \eqref{eq21} and Lemma \ref{vol}/(i),(ii) we have 
\be\label{eq27}
C(r,d) \geq 
2^d \, r\, \frac{(\ln r - d\ln 2)^{d-1}}{(d-1)!}\cdot\Big(1- \frac{d-1}{\ln r - d\ln 2}\Big)
\ee
for $r\ge 2^d$. Next we need a proper upper estimate for $\ln C(r,d)$.
In fact, if $r\ge e^{d-1}\geq e^{\ell}$ we have 
\[
\frac{(\ln r)^{\ell-1}}{(\ell-1)!} \le  \frac{(\ln r)^d}{d!}\, . 
\]
 Hence, using Stirling's formula, we can estimate
\be\label{ws-10}
    C(r,d) \le 1+  \sum\limits_{\ell = 1}^d \binom{d}{\ell}\, 2^\ell \, r\,  \frac{(\ln r)^{d}}{d!} = 3^d\, r\,\frac{(\ln r)^{d}}{d!}
\leq r \, \left(\frac{3e \ln r}{d}\right)^{d}\,.
\ee
This gives, for $r\ge e^{d-1}$\,, 
\be
\ln C(r,d)\le \ln r +d \ln\ln r +d\ln(3e/d)
\ee

Now let $C(r,d)<n\le C(r+1,d)$. Then we have $a_n=\frac 1{r+1}$, and
inserting the above inequalities in \eqref{eq27} yields  
\be\nonumber
    \begin{split}
        \frac{n a_n}{(\ln r)^{d-1}}&\ge\frac{C(r,d)}{(r+1)(\ln C(r,d))^{d-1}} \\
&\geq \frac{2^d}{(d-1)!} 
\cdot\frac{r}{r+1}\cdot\left(\frac{\ln r -d\ln 2}{\ln r +d \ln\ln r +d\ln(3e/d)}\right)^{d-1}
\cdot \left(1-\frac{d-1}{\ln r-d\ln 2}\right)\\
        &\xrightarrow[r\to\infty]{} \frac{2^d}{(d-1)!}\,.
    \end{split}
\ee
This implies
$$
    \liminf\limits_{n\to\infty}\frac{n\,a_n}{(\ln n)^{d-1}} \geq \frac{2^d}{(d-1)!}\,,
$$
and the proof is complete.
\eproof

\begin{rem}\rm 
(i) For $d=1$,  the mixed space $H^{s,\#}_{\mix} (\tor)$ coincides with the 
isotropic space $H^{s,\#}(\tor)$, with equality of the corresponding norms.
Then
$$
 \lim_{n \to \infty} \, n^s\,a_{n} (I_d: \, H^{s,\#} (\tor) \to L_2 (\tor)) =
 2^{s} 
 $$
follows, a result, already proved in \cite{KSU1}.
\\
(ii) As a consequence of Stirling's formula we observe that
$$
      \Big[\frac{2^{d}}{(d-1)!}\Big]^s \asymp \Big(\frac{d}{2\pi}\Big)^{s/2}\Big(\frac{2e}{d}\Big)^{ds}\,,
$$
where $a_n \asymp b_n$ means $\lim_{n\to \infty} a_n/b_n =1$. This shows
a super-exponential decay of the constant.
\end{rem}

\noindent 
Being interested in explicit constants $c_s (d), C_s (d)$ in (\ref{ws-06}), we can learn something from Theorem \ref{lim1}.
Fix $d\in\N$ and $s>0$. Then for any given $\varepsilon >0$ there exists $n_0=n_0(\varepsilon)\in \N $ such that
\[
\Big[\frac{2^d}{(d-1)!}\Big]^s - \varepsilon \le \frac{n^s\,a_n (I_d: \, H^{s,\#}_{\mix}(\T) \to L_2(\T))}{(\ln n)^{s(d-1)}} 
\le \Big[\frac{2^d}{(d-1)!}\Big]^s + \varepsilon\qquad\text{for all } n\ge n_0\,.
\]
Equivalently, for any given $n_1\in\N$ there is a constant $\lambda=\lambda (n_1)$, $1<\lambda <\infty$,
such that 
\[
\frac 1\lambda\Big[\frac{2^d}{(d-1)!}\Big]^s \le \frac{n^s\,a_n (I_d: \, H^{s,\#}_{\mix}(\T) \to L_2(\T))}{(\ln n)^{s(d-1)}} 
\le \lambda\Big[\frac{2^d}{(d-1)!}\Big]^s\,  \qquad \text{for all } n \ge n_1\, . 
\]
We now aim at controlling the constant $\lambda (n_1)$ 
for certain (large) values of $n_1$.

\begin{satz}\label{gut}
Let $s>0$ and  $d\in \N$.\\
(i) Then
 $$
  a_{n} (I_d: \, H^{s,\#}_{\mix} (\T) \to L_2 (\T)) \le
\Big[\frac{3^d}{(d-1)!}\Big]^s\,\frac{(\ln n)^{(d-1)s}}{n^s} \, \qquad \mbox{if}\quad n \ge  27^d \, .
 $$
 (ii) On the other hand,
 $$
  a_{n} (I_d: \, H^{s,\#}_{\mix} (\T) \to L_2 (\T)) \ge 
  \left[
  \frac{3}{d!}\, \left(\frac{2}{2 + \ln 12}\right)^{d}\right]^s\, \frac{(\ln n)^{(d-1)s}}{n^s} 
   \, 
\qquad 
\mbox{if}\quad n > (12\, e^2)^d \, .
 $$
\end{satz}

\bproof Again, it is enough to deal with the case $s=1$.\\ For $r\in \N$ and $C(r-1,d) < n \leq C(r,d)$ we
have $a_n = 1/r$\,, whence $a_n \le 1$ for all $n$. 
\\\newline
{\em Step 1.} Proof of (i). First recall that $C(r,d) \geq C(r,1) \geq r$ (see the previous proof), and that $n/(\ln n)^{d-1}$ is increasing for $n >e^{d-1}$.
Similarly as above in (\ref{ws-10}) we have, for all $n >e^{d-1}$,
\[
\frac{n\,  a_n}{(\ln n)^{d-1}} \leq \frac{C(r,d)}{r(\ln C(r,d))^{d-1}}
\leq \frac{1+\sum\limits_{\ell = 1}^{d} 2^{\ell} \binom{d}{\ell} v_{\ell}(r)}{r(\ln r)^{d-1}}
\leq \frac{1 + \sum\limits_{\ell = 1}^{d} 2^{\ell}\,  \binom{d}{\ell} \, \frac{(\ln r)^{\ell-1}}{(\ell-1)!}}{(\ln r)^{d-1}}\,.
\]
Since $\frac{(\ln r)^{\ell-1}}{(\ell-1)!}\le \frac{(\ln r)^{d-1}}{(d-1)!}$\,, we have
$$
    \sup\limits_{n \ge C(e^d,d)} \frac{n\, a_n}{(\ln n)^{d-1}} \leq \frac{1}{(d-1)!} \sum\limits_{\ell = 0}^d 2^{\ell} \binom{d}{\ell} = \frac{3^d}{(d-1)!}\,.
$$
Next we give a precise range for $n$ in which this estimate holds. To this end, we estimate $C(r,d)$ with $r=e^d$ from above.
The obvious inequality $x^k/k!\le e^x$ applied to $x=\ln r=d$ gives
\be\label{eq30}
       C(e^d,d) \leq  r + \sum\limits_{\ell = 1}^d 2^{\ell} \binom{d}{\ell} r\frac{(\ln r)^{\ell-1}}{(\ell-1)!} \leq 3^d\,r^2= (3e^2)^d\,.
\ee
This shows that (i) holds for all $n>(3e^2)^d$. 

Finally we show that for $n\ge 27^d$ the upper bound in (i) is non-trivial (i.e $<1$). 
To see this, we use $(\ln n)^{d-1}/(d-1)!\le (\ln n)^d/d!$ for $n >e^d$, recall that 
the function $f_d (t) = t^{-1}\, (\ln t)^{d-1}$ is decreasing on $[e^{d-1}, \infty)$.
Applying Stirling's formula and these monotonicity assertions, estimate (i) yields 
$$
a_n\le \frac{(3\ln n)^d}{d!\,n}\le\left(\frac{3ed\ln 27}{27d}\right)^d=\left(\frac{e\ln 3}{3}\right)^d\,. 
$$
Since $(e\ln 3)/3=0.99544...<1$, we see that indeed $a_n<1$.
\\
\newline
{\em Step 2.} Let us turn to the estimate from below. 
Arguing as in (\ref{eq30}) we find
\be\label{ws-25}
\ln C(r,d) \le \ln (3^d\,r^2) \, , \qquad r \ge e^d\, .
\ee
Next we estimate of $C(r,d)$ from below. 
We start with formula \eqref{eq27}
$$
C(r,d) \geq 
\frac{2^d \, r}{(d-1)!}\, \left(\ln (r/2^{d}) \right)^{d-1} \, \left(1- \frac{d-1}{\ln (r/2^{d})}\right)\, .
$$
For $r\ge r_0 := (2e)^d$ and $C(r,d)< n \le C(r+1,d)$, using again the monotonicity of $f_d$, this implies
\begin{eqnarray}\label{t-2}
 \frac{n \, a_n}{(\ln n)^{d-1}} & \ge & \frac{C(r,d)}{(r+1)\, (\ln C(r,d))^{d-1}}
\\
& \ge & \frac{2^d}{(d-1)!} \cdot \frac{r}{r+1} \cdot \left(\frac{\ln (r/2^{d})}{\ln (3^d\, r^2)}\right)^{d-1}
\cdot \left(1- \frac{d-1}{\ln (r/2^d)}\right)\, .
\nonumber
\end{eqnarray}
Concerning the different factors we have, for all $r\ge r_0$\,,
\beqq
\frac{r}{r+1} & \ge & \frac{5^d}{5^d + 1} \ge \frac 56 \, ,\quad\text{since } 2e>5\,,
\\
\frac{\ln (r/2^{d})}{\ln (3^d\, r^2)} & \ge & \frac{\ln (e^d)}{\ln (12^d \, e^{2d})} = 
\frac{1}{2 + \ln 12} 
\,
\\
1- \frac{d-1}{\ln (r/2^d)} &\ge& 1-\frac{d-1}{d} = \frac 1d\, , \\
\eeqq
Hence, taking $\frac{5}{6}\,(2+\ln 12)\ge 3$ into account, we arrive at 
\[
\sup_{n \ge C(r_0,d)}\, \frac{n \, a_n}{(\ln n)^{d-1}} \ge \frac{3}{d!}\,\left(\frac{2}{2 +\ln 12}\right)^d \, .
\]
Since $C(r_0,d) \le 3^d \, r_0^2 = (12 \, e^2)^d$, the proof of (ii) is finished.
\eproof

\begin{rem}
 \rm
(i) We can improve on the bound in (ii), if we choose a larger value of $r_0$. But then the range of $n$ for which (ii) holds becomes smaller.
Here is an example. Taking $r_1 := 48^{d/2}>r_0$, we get
\[
\frac{\ln (r_1/2^d)}{\ln (3^d\, r^2_1)} = \frac 14 \, , \quad 
\frac{r}{r+1}  \ge \frac{\sqrt{48}}{\sqrt{48}+1}\ge \frac{41}{47}\, 
\quad \mbox{and}\quad 1- \frac{d-1}{\ln (r/2^d)} \ge 1-\frac{1}{\ln\sqrt{12}}= 0.766173...\ge \frac 34\, .
\]   
Since $\frac{41}{47}\cdot \frac 43 >1$, we obtain
$$
  a_{n} (I_d: \, H^{s,\#}_{\mix} (\T) \to L_2 (\T)) \ge \left[\frac{3^d(\ln n)^{d-1}}{2^d\,(d-1)!\,n}\right]^s
\qquad 
\mbox{if}\quad n > (48)^{d/2} \, .
$$
(ii) Conversely, one can extend the range of $n$ in (ii) by making $r_0$ smaller. However, this strategy is limited by our method.
Indeed, if $r \le 2^d \, e^{d-1}$, then for the last factor in \eqref{t-2} we have 
\[
1- \frac{d-1}{\ln (r/2^d)} \le 1- \frac{d-1}{\ln (e^{d-1})} =0
\]
and our estimate \eqref{t-2} becomes useless.
\end{rem}


\subsection*{Some ``local'' improvements}


We do not claim that the estimates obtained in Theorem \ref{gut} are optimal in $d$ and $n$. 
They can be improved in various ways. But these improvements take place only locally, i.e., 
for $n$ taken from a finite interval. 
\\
Let $d\in\N$, and let $(\sigma_n)_{n\in\N}$ be the non-increasing rearrangement of 
$\left(1/w_{s}^\# (k)\right)_{k\in\Z}$\,.
Now we estimate $\sigma_n$
by a tensor trick. This method is very simple and works for any $d\in\N$. The best result that can be obtained 
in this way differs by a log-factor from the exact asymptotic order of $\sigma_n$. However, since the resulting 
constants are quite explicit, it improves on Theorem\ \ref{gut}, (i)
if $15^d < n < \exp(\sqrt{d/(2\pi)}\cdot 1.5^d)$, see Remark \ref{localimp} below. 

\begin{lem}\label{moderate}
For every $d\in \N$, every $s>0$ and all $n\ge 15^d$ it holds
\be\label{ws-13}
a_n(I_d: H^{s,\#}_{\mix}(\T)\to L_2(\T)) \le \frac{1}{n^s} \left(\frac{2e\ln n}{d}\right)^{sd}\,.
\ee
\end{lem}

\begin{proof}
Again we concentrate on $s=1$.
For arbitrary $p>1$ we have
\begin{align*}
n \sigma_n^p&\le \sum_{j=1}^n \sigma_j^p\le \sum_{j=1}^\infty \sigma_j^p =\sum_{k\in\Z} 
\prod_{\ell =1}^d \, \frac{1}{(1+|k_\ell|)^p}= \left( \sum_{m\in\zz} \frac{1}{(1+|m|)^p} \right)^d
=\left(1+2\sum_{m=2}^\infty \frac{1}{m^p}\right)^d\\
&\le \left(1+2\int_1^\infty\frac{dx}{x^p}  \right)^d=\left(1+\frac{2}{p-1} \right)^d=\left(\frac{p+1}{p-1} \right)^d\,,
\end{align*}
which implies 
$$
\sigma_n\le n^{-1/p}\left(\frac{p+1}{p-1} \right)^{d/p}\le  n^{-1/p}\left(\frac{p+1}{p-1} \right)^{d}\,.
$$
Now we optimize, for given $n\in\N$, over the free parameter $p>1$. 
Note that the map $p\mapsto \frac{p+1}{p-1}$ is a bijection from the interval $(1,\infty)$ onto itself. If $n>e^{d/2}$, 
we have  $(2\ln n)/d>1$, and so we can choose $p>1$ such that
\be\label{ws-31}
\frac{p+1}{p-1}=\frac{2\ln n}{d}\,.
\ee
It remains to estimate the exponent in $n^{-1/p}$. We have
$$
-\frac 1p=-1+\frac{p-1}{p+1}\cdot \frac{p+1}{p}\le -1+\frac{d}{\ln n}\,,
$$
and hence $$n^{-1/p}\le n^{-1}\cdot e^d\,.$$ 
This implies the desired estimate
$$
\sigma_n\le \frac 1n\left(\frac{2e\ln n}{d}\right)^d\qquad\text{for all } n>e^{d/2}\,.
$$
This bound is non-trivial (i.e. $<1$) for $n\ge 15^d$. 
\end{proof}

\begin{rem}\label{localimp}
 \rm
(i) In our  Theorem\ \ref{gut} we got the upper bound (in slightly rewritten form)
$$
a_n(I_d: H^{s,\#}_{\mix}(\T)\to L_2(\T)) \le \frac{1}{n^s\ln^s n}\, \cdot\, \Big(\sqrt{\frac{d}{2\pi}}\, \frac{(3e\ln n)^d}{d^d}\Big)^s \, .
$$
This bound is larger than the bound obtained in (\ref{ws-13}), if and only if 
$$n\le \exp(\sqrt{d/2\pi}\cdot 1.5^d)$$ 
which is doubly exponential in $d$, that
means far beyond all $n$ in `real life' applications or in numerical analysis.
So the tensor trick might after all be quite useful, although it cannot give the exact asymptotic rate.
\\
(ii) The first part of this remark explains that the choice of $p$ in (\ref{ws-31}) is reasonable, since it almost gives the exact asymptotic rate as $n\to\infty$.
However, it is not optimal for all $n$.
This might be seen as follows.
We simply fix $p$ from the very beginning and follow the above argument.
The most simple choice is $p=2$. Then we have the exact value of the sum $\sum_{j=1}^\infty \sigma_j^2$ at hand and the outcome is 
\be\label{ws-14}
a_n(I_d: H^{s,\#}_{\mix}(\T)\to L_2(\T)) \le  \left(\frac{1}{n} \, \Big(\frac{\pi^2}{3}-1\Big)^d\right)^{s/2}\,,
\ee
for all $n \in \N$.
Since $2 < \frac{\pi^2}{3}-1< e$ this estimate is of certain use if $n\ge e^d$. Now we compare 
(\ref{ws-13}) and (\ref{ws-14}). It follows
\[
\left(\frac{1}{n} \, \Big(\frac{\pi^2}{3}-1\Big)^d\right)^{s/2}
\le \frac{1}{n^s} \left(\frac{2e\ln n}{d}\right)^{sd}
\]
if and only if 
\[
\frac{n^{1/(2d)}}{\ln n} \, \le \frac{2e}{d} \Big(\frac{\pi^2}{3}-1\Big)^{-1/2} \, .
\]
A sufficient condition is given by 
\[
\frac{n^{1/(2d)}}{\ln n} \, \le \frac{e}{d} \, .
\]
The function $f(x):= x^{1/(2d)}/\ln x$ is decreasing on $[1,e^{2d}]$ and increasing on $[e^{2d}, \infty)$,
and $f(e^d)= \sqrt{e}/d < e/d$.
Because of $f(e^{cd})= (\sqrt{e})^c/(cd) \le  e/d$ if and only if $c - 2\, \ln c \le 2 $
we conclude that 
\[
\left(\frac{1}{n} \, \Big(\frac{\pi^2}{3}-1\Big)^d\right)^{s/2}
\le \frac{1}{n^s} \left(\frac{2e\ln n}{d}\right)^{sd} \qquad \mbox{if} \quad e^{d}\le n \le e^{c_0d},
\]
where $c_0$ is the solution of $c - 2\, \ln c = 2 $ ($5.35 < c_0 < 5.36$).
Hence, (\ref{ws-14}) is better than (\ref{ws-13}) as long as $e^{d}\le n \le e^{c_0d}$.
Different choices of $p$ (e.g., $p=3/2$, $p=4$) lead to different intervals of optimality, we omit further details.
\end{rem}


\subsection{The approximation numbers $a_n(I_d:H^{s,\#}_{\mix}(\T)\to L_2 (\T)$ }


For computational issues, the number $(3e^2)^d$ in Theorem\ \ref{gut} might be much too large. 
We will now focus on estimates for smaller $n$
and investigate the so-called preasymptotic behavior. To be more precise, we will deal with estimates 
of $a_n(I_d: \, H^{s,\#}_{\mix}(\T)\to L_2 (\T))$
from above and below in the range $1 \le n \le (d/2)\, 4^d$.

\begin{satz}\label{small}
Let $s>0$ and  $d\in \N$, $d\ge 2$. 
For all $1 \leq n \leq \frac{d}{2}4^d$ it holds
\be\label{ws-46}
a_n(I_d:H^{s,\#}_{\mix}(\T) \to L_2(\T)) \leq \Big(\frac{e^2}{n}\Big)^{\frac{s}{2+\log_2 d}}\,.
\ee
\end{satz}

\bproof
It is enough to consider the case $s=1$. Let $1\le r \leq 2^d$. 
Then for $C(r-1,d) < n \leq C(r,d)$ we have $a_n = 1/r$\,. Let us estimate
$C(r,d)$ in this case. We shall use $[x]$ to denote the greatest integer part of the real number $x$. 
Starting from \eqref{eq21} and using the obvious estimate 
$x^k/k!\le e^x$ applied to $x=\ln r=d$, we obtain  
\be\nonumber
    \begin{split}
        C(r,d) &= 1+\sum\limits_{\ell = 1}^{[\log_2 r]} 2^{\ell}\, \binom{d}{\ell}\, A(r,\ell)
        \leq 1+\sum\limits_{\ell = 1}^{[\log_2 r] } 2^{\ell}\, \binom{d}{\ell}\, v_{\ell}(r)
        \leq 1+\sum\limits_{\ell = 1}^{[\log_2 r]} 2^{\ell}\, \binom{d}{\ell}\, r\, \frac{(\ln r)^{\ell-1}}{(\ell-1)!}\\
        &\leq r^2\sum\limits_{\ell = 0}^{[\log_2 r]} 2^{\ell} \, \binom{d}{\ell} \, 
\leq r^2\sum\limits_{\ell=0}^{[\log_2 r]} 2^{\ell} \, \frac{d^\ell}{\ell !} 
        \leq r^2 \, d^{[\log_2 r]}\, e^2\leq e^2\, r^{2+\log_2 d}\,.
    \end{split}
\ee
This gives $n \leq C(r,d) \leq e^2\, r^{2+\log_2 d}$ which implies $1/r \leq (e^2/n)^{1/(2+\log_2 d)}$\,. 
Therefore we get for all $n \leq C(r,d)$ the relation
$$
        a_n \leq \Big(\frac{e^2}{n}\Big)^{\frac{1}{2+\log_2 d}}\,.
$$
This estimate holds for all $n\leq C(2^d,d)$. 
To estimate $C(2^d,d)$ from below we need a preparation. 
Obviously, in case $\ell \ge 2$, we have
\[
\Big\{k \in \N^\ell: \: k_2 = \, \ldots \, =k_\ell = 1 \, ,
\quad (1+k_1)\, 2^{\ell-1}\le r \Big\} \subset \cm (r,\ell) \, .
\]
The set of the left hand side has cardinality $[r2^{-\ell+1}]-1$.
By interchanging the roles of $k_1$ with $k_j$, $j=2,3,\ldots$, we find $\ell$ subsets of $\cm (r,\ell)$
having only $(1, \ldots \, , 1)$ in the intersection. This implies
\begin{equation}\label{er-01}
  A(r,\ell) \ge \ell \, [r\, 2^{-\ell+1}]-2\ell+1\, , 
\end{equation}
which is also true for $\ell =1$.
In case  $r=2 ^d$ we obtain from Lemma \ref{ook}

\begin{eqnarray}
   C(2^d,d) & = &  1+\sum\limits_{\ell = 1}^{d} 2^{\ell}\, \binom{d}{\ell}\, A(2^d,\ell) 
\geq 1+\sum\limits_{\ell = 1}^{d} 2^{\ell}\, \binom{d}{\ell} \, \Big(\ell 2^{d-\ell +1} - 2\ell +1\Big)
\nonumber
\\
& = &  2^{d+1} \, \sum\limits_{\ell = 1}^{d}  \binom{d}{\ell} \, \ell \,  
-2 \sum\limits_{\ell = 1}^{d}  \binom{d}{\ell} \, \ell \, 2^\ell
+ 3^d 
\nonumber
\\
&  = & 3^d + 2^{d+1}
\sum\limits_{\ell = 1}^d \frac{d!}{(d-\ell)!(\ell-1)!}-2\sum\limits_{\ell = 1}^d 2^{\ell}\frac{d!}{(d-\ell)!(\ell-1)!}
\nonumber\\
&=& 3^d +2^{d+1}d\sum\limits_{\ell = 1}^d \binom{d-1}{\ell-1}-2d\sum\limits_{\ell = 1}^d 2^{\ell}\binom{d-1}{\ell-1}
\nonumber\\
&=& 3^d +2^{d+1}d2^{d-1}-4d3^{d-1}
\nonumber\\
&=& 3^d +d4^d-4d3^{d-1}\,\nonumber.
\end{eqnarray}
Hence, we have $C(2^d,d) \geq 3^d +d4^d-\frac{4}{3}d3^d$\,. 
Note, that 
\begin{eqnarray}
  C(2^d,d) \geq \frac{d}{2}4^d &\Longleftrightarrow& 3^d +d4^d-\frac{4}{3}d3^d \geq \frac{d}{2}4^d
  \nonumber\\
  &\Longleftrightarrow& \frac{d}{2}4^d \geq  \Big(\frac{4}{3}d-1\Big)3^d\,.\label{4.19}
\end{eqnarray}
Of course, \eqref{4.19} is true for all $d\geq 2$\,. The proof is complete. 
\eproof

Let us turn to an estimate from below. 

\begin{satz}\label{lowersmall} Let $d\in \N$, $d\ge 2$, and $s>0$. 
For $n \ge 2$ we define
\be\label{er-04}
\alpha (n,d) := 2+\log_2\Big(\frac{d}{\log_2 n} + \frac 12\Big) \, .
\ee
For all $2 \leq n \le \frac{d}{2}4^d$ it holds 
\be\label{f64}
    a_n(I_d:H^{s,\#}_{\mix}(\T) \to L_2(\T)) \geq 2^{-s} n^{-\frac{s}{\alpha (n,d)}}\,.
\ee
\end{satz}

\bproof It suffices to deal with the case $s=1$. Let $2\leq r\leq 2^d$ such that $C(r-1,d) < n \leq C(r,d)$. 
Furthermore, let $m \in \N_0$ be determined from
\be\label{er-02}
2^m \le r-1 < 2^{m+1}\,.
\ee
Then \eqref{eq21} and \eqref{er-01} imply 
$$
  n>C(r-1,d) = 1+ \sum\limits_{\ell = 1}^m 2^{\ell}\binom{d}{\ell}A(r,\ell) \ge 
\sum\limits_{\ell = 1}^m 2^{\ell}\binom{d}{\ell} \, \Big(\ell \, 2^{m-\ell+1} -\ell +1\Big)
\ge 2^{m+1}\, \sum\limits_{\ell = 1}^m \binom{d}{\ell}
\,.
$$
Hence
$$
  n> 2^{m+1} \, \sum\limits_{\ell=1}^{m}\binom{m}{\ell}\frac{d(d-1)\cdot(d-\ell+1)}{m(m-1)\cdots(m-\ell+1)} 
\geq  2^{m+1}  \, \sum\limits_{\ell = 1}^m \binom{m}{\ell}\Big(\frac{d}{m}\Big)^\ell\,.
$$
Taking the binomial formula into account, this implies 
$$
    n> 2^{m+1} \, \Big\{\Big(1+ \frac{d}{m}\Big)^m-1\Big\} \ge   4^m\Big(\frac{d+m}{2m}\Big)^m\,.
$$
Next we apply $\log_2$ on both sides and obtain
\be\label{f81}
  \log_2 n > m \, \Big\{2+\log_2\Big(\frac{d}{2m} + \frac 12\Big)\Big\}\ge 2m\,,
\ee
since $\frac{d}{2m} + \frac 12\ge 1$\,.  Together with \eqref{f81} this yields
\[
   \log_2 n > m\, \Big\{2+\log_2\Big(\frac{d}{\log_2 n} + \frac 12\Big)\Big\} = m \cdot \alpha (n,d)  \,.
\]
Rewriting this inequality we get
\[
2^m < 2^{\frac{\log_2 n}{\alpha (n,d)}} = n^{\frac{1}{\alpha (n,d)}}\, .  
\]
Taking \eqref{er-02} into account, we finally conclude
\be\label{f66}
a_n = \frac 1r \geq \frac{1}{2^{m+1}}
\geq \frac{1}{2}\, n^{-\frac{1}{\alpha (n,d)}}
\ee
for all $n$, $C(1,d) < n \le C (2^d,d)$, hence, at least for $2 \le n \le \frac{d}{2}4^d$ (see \eqref{4.19}).
If $s \neq 1$ we obtain \eqref{f64} by raising \eqref{f66} to the power $s$. This finishes the proof. 
\eproof

\begin{rem}
 \rm
(i)
Of course, there remains a gap between the lower bound in \eqref{f64} and the upper bound in \eqref{ws-46}.
For simplicity we comment on this gap for $s=1$ only.
In fact, we have
\[
0 <  \frac{1}{\alpha (n,d)} - \frac{1}{2 + \log_2 d} < \frac{\log_2\log_2 n}{4 + 2\log_2 d} \, .
\]
The gap is very mildly growing in $n$ (keeping $d$ fixed).
Therefore, our estimates are loosing quality when $n$ increases.
Right now we do not have a conjecture about the correct bounds, most probably 
both, the lower and the upper estimate, can be improved.
 \\
(ii) Note that $\alpha(n,d)$ is decreasing in $n$. Hence, on certain smaller intervals of $n$, the dependence on $n$ in $\alpha(n,d)$ can be removed by simple monotonicity arguments. 
For instance, since $\alpha(4^d,d)=2$ and $\alpha(2^{2d/3},d)=3$, we get
\[
a_n(I_d:H^{s,\#}_{\mix}(\T) \to L_2(\T)) \geq 2^{-s} n^{-s/2} 
\]
simultaneously for all $2 \le n \le 4^d$, and
\[
a_n(I_d:H^{s,\#}_{\mix}(\T) \to L_2(\T)) \geq 2^{-s} n^{-s/3} 
\]
simultaneously for all $2 \le n \le 2^{2d/3}$.\\
The approximation rate in these examples is much worse than the asymptotic rate $n^{-s}$ (ignoring the logarithmic factors).
This illustrates well that one has to wait exponentially long until one can "see" the correct asymptotic behavior of the approximation numbers.

\end{rem}


\subsection{The approximation numbers of $H^{s,+}_{\mix}(\T)$, 
$H^{s,*}_{\mix}(\T)$, and $H^m_{\mix}(\T)$ in $L_2 (\T)$}


Now we turn to the investigation of 
$a_{n} (I_d: \, H^{s,+}_{\mix} (\T) \to L_2 (\T))$ and 
$a_{n} (I_d: \, H^{s,*}_{\mix} (\T) \to L_2 (\T))$.


\subsubsection{Preparation}


For $s>0$ and $k\in\Z$ we put
\be
 w_{s}^+(k) :=   \prod\limits_{j=1}^d(1+ |k_j|^2)^{s/2}\,  \qquad \mbox{and}\qquad
 w_{s}^*(k) :=   \prod\limits_{j=1}^d(1+ |k_j|^{2s})^{1/2}\,,
\ee
see (\ref{norm4.5}) and (\ref{norm5}), respectively.
Of interest for us are the non-increasing rearrangements of $(1/w_{s}^+ (k))_{k \in \Z}$ and 
$(1/w_{s}^*(k))_{k \in \Z}$. 
Let 
\[
C^+_s(r,d)  :=  \#\{k\in \Z~:~w_{s,+}(k) \leq  r\}\qquad
\mbox{and}\qquad
C^*_s(r,d)  :=  \#\{k\in \Z~:~w_{s,*}(k) \leq  r\}\, , 
\]
where $r\ge 1$, $s>0$ are  real numbers.
Let us also define the smaller numbers 
\[
c^+_s(r,d)  :=  \#\{k\in \Z~:~w_{s,+}(k) <  r\}\qquad 
\mbox{and} \qquad
c^*_s(r,d)  :=  \#\{k\in \Z~:~w_{s,*}(k) <  r\}\, . 
\]
In contrast to the weights $w_s^\#$, we now have no 
complete overview over all possible values of $w_{s}^+(k)$, 
$w_{s}^*(k)$ as $k$ runs through $\Z$. Therefore, it is impossible to describe the full sequence of 
approximation numbers $a_n$. However, since $C^+_s((1+r^2)^{s/2},d) > c^+_s((1+r^2)^{s/2},d)$ and 
$C^{*}_s((1+r^{2s})^{1/2},d)>c^{*}_s((1+r^{2s})^{1/2},d)$ if $r\in \n$, 
we have at least some partial information about the piecewise constant sequence $a_n$ of approximation numbers.

\begin{lem}
Let $s>0$ and $r\in\N_0$.\\
{\rm (i)} If $c^+_s((1+r^2)^{s/2},d) < n\leq C^+_s((1+r^2)^{s/2},d)$, then
\be\label{f61}
a_{n} (I_d: \, H^{s,+}_{\mix} (\T) \to L_2 (\T)) = (1+r^2)^{-s/2}\, .
\ee
{\rm (ii)} If $c^+_s((1+r^{2s})^{1/2},d) < n\leq C^+_s((1+r^{2s})^{1/2},d)$, then 
$$
a_{n} (I_d: \, H^{s,*}_{\mix} (\T) \to L_2 (\T)) = (1+r^{2s})^{-1/2}\, .
$$
\end{lem}
\bproof
By Lemma\ \ref{sing} and the same principles as used in \eqref{eq072},
it is enough to note that we have $w_{s}^+(k)= (1+r^2)^{s/2}$ and 
$ w_{s}^*(k)\ge (1+r^{2s})^{1/2}$ for $k=(r,0,\ldots,0)\in\Z$ .
\eproof


\subsubsection{Some more combinatorics}


Since we have only incomplete information on the set of all values attained by the weights $w_{s}^+(k)$ and 
$ w_{s}^*(k)$, it is very difficult to establish similar combinatorial identities and sharp estimates as 
for the weight $ w_{s}^\#(k)$. Therefore we decided for a different strategy.
For $\ell\in\Z$, $0<\eps\le 1$ and $d\in\N$ let
$$
a_\ell:=\frac{1}{1+|\ell|}\quad,\quad\mathcal{A}_d(\eps) :=\left\{k\in\Z:\: a_{k_1}\cdots a_{k_d}\ge \eps \right\}\quad,\quad
A_d(\eps):= \# \mathcal{A}_d(\eps)\,.
$$
Because of
$$
\Big\{k\in\Z:\: a_{k_1}\cdots a_{k_d}\ge \frac 1r \Big\}  = \Big\{k\in\Z:\: \prod_{j=1}^d (1+ |k_j|) \le r \Big\} = \cn (r,d)
$$
we have $A_d(1/r)=C(r,d)$ for all $r\in \N_0$. Using \eqref{eq28}, Lemma\ \ref{combi} and a simple monotonicity argument, this implies
\begin{equation}\label{ws-19}
  \lim_{\eps\downarrow 0}\frac{\eps\cdot A_d(\eps)}{(\ln A_d(\eps))^{d-1}} = \lim_{r \to \infty} \frac{C(r,d)}{r(\ln C(r,d))^{d-1}} = \frac{2^d}{(d-1)!} \,.
\end{equation}
As consequences of these identities, we find
for arbitrary $\lambda >0$ and all $d\in\N$
\begin{equation}\label{ad}
\lim_{\eps\downarrow 0}\frac{A_d(\eps)}{A_d(\lambda\eps)}=\lambda\qquad\text{and}\qquad 
\lim_{\eps\downarrow 0}\frac{A_{d-1}(\eps)}{A_d(\lambda\eps)}= 0\,.
\end{equation}

\begin{lem}\label{combi} Let $(b_\ell)_{\ell\in\zz}$ be a sequence indexed by $\zz$ such that 
$$
0<b_{\ell}\le b_0=1 \quad\text{for all } \ell\neq 0\qquad\text{and }\quad\lim_{|\ell|\to\infty}\frac{a_\ell}{b_\ell}=1\,.
$$
Similarly as for $(a_\ell)_{\ell\in\Z}$ we define
$\mathcal{B}_d(\eps)$ and $B_d(\eps)$ associated to $(b_\ell)_{\ell\in\Z}$\,.
Then we have 
\be\label{Lem412}
\lim_{\eps\downarrow 0}\frac{B_d(\eps)}{A_d(\eps)}=1\,.
\ee
\end{lem}

\bproof Let us first observe that there are constants $0<c\le C<\infty$ such that $c\le a_{\ell}/b_{\ell} \le C$ for all $\ell\in \zz$.
Fix now $0<\eps\le 1$ and $\delta >0$ (small), and select $m=m(\delta)\in\N$ such that
$$
1-\delta\le \frac{a_{\ell}}{b_{\ell}}\le 1+\delta \qquad\text{ for all } |\ell|\ge m\,.
$$
For $k\in\mathcal{B}_d(\eps)$, we distinguish two cases.
\\\\
Case 1, $|k_j|\ge m$ for all $j$. This implies
$$
\eps\le \prod_{j=1}^d b_{k_j}\le \prod_{j=1}^d \frac{a_{k_j}}{1-\delta}\,,
$$
and thus $k\in\mathcal{A}_d((1-\delta)^d\eps)$\,.
\\\\
Case 2,  $|k_\ell|<m$ for some $\ell$. Now we have
$$
\eps\le \prod_{j=1}^d b_{k_j}=b_{k_\ell}\prod_{j\neq\ell} b_{k_j}\le \prod_{j\neq\ell} \frac{a_{k_j}}{c}\,,
$$
which gives $(k_1,...,k_{\ell-1},k_{\ell+1},...,k_d)\in \mathcal{A}_{d-1}(c^{d-1}\eps)$.
Since there are $d$ choices of the index $\ell\in\{1,...,d\}$ and $2m-1$ possible values of $k_\ell$, we conclude that
$$
B_d(\eps)\le A_d((1-\delta)^d\eps) + (2m-1)d\cdot A_{d-1}(c^{d-1}\eps)\,.
$$
Using the relations (\ref{ad}), this gives  
\begin{equation}\label{limsup}
\limsup_{\eps\downarrow 0}\frac{B_d(\eps)}{A_d(\eps)}\le \frac{1}{(1-\delta)^d}\,.
\end{equation}
Now we show a lower estimate for $B_d(\eps)$. Let $k\in \mathcal{A}_d((1+\delta)^d\eps)$. Again we distinguish two cases.\\\\
If all $|k_j|\ge m$,  we have
$$
(1+\delta)^d\eps\le \prod_{j=1}^d a_{k_j}\le \prod_{j=1}^d (1+\delta)b_{k_j}\,,
$$
that means $k\in\mathcal{B}_d(\eps)$.
\\\\
Otherwise, if $k_\ell <m$ for some $\ell$, we have
$$
(1+\delta)^d\eps\le a_{k_\ell}\cdot\prod_{j\neq \ell} a_{k_j}\le \prod_{j\neq \ell} a_{k_j}\,,
$$
which means $(k_1,...,k_{\ell-1},k_{\ell+1},...,k_d)\in \mathcal{A}_{d-1}((1+\delta)^d\eps)$, and we get
$$
A_d((1+\delta)^d\eps)-2(m-1)d\cdot A_{d-1}((1+\delta)^d\eps)\le B_d(\eps)\,.
$$
This implies, using again (\ref{ad}),
\begin{equation}\label{liminf}
\liminf_{\eps\downarrow 0} \frac{B_d(\eps)}{A_d(\eps)}\ge\frac{1}{(1+\delta)^d}\,,
\end{equation}
and since (\ref{limsup}) and (\ref{liminf}) are true for all $\delta>0$, the proof is finished.
\eproof

\noindent 
There are some simple consequences of  Lemma\ \ref{combi} which are of interest for us.
Taking logarithms  in \eqref{Lem412} yields 
$$
    \lim\limits_{\varepsilon \downarrow 0} \Big(\ln B_d(\eps) - \ln A_d(\eps)\Big) = 0.
$$
Since $\lim_{\eps \downarrow 0} A_d(\eps) = \infty$\,, we get
\be\label{f80}
  \lim\limits_{\eps\downarrow 0}\frac{\ln B_d(\eps)}{\ln A_d(\eps)} = \lim\limits_{\eps \downarrow 0}\frac{\ln B_d(\eps)-\ln A_d(\eps)}{\ln A_d(\eps)}+1 = 1\,.
\ee
Hence, 
$$
\frac{\eps\cdot B_d(\eps)}{(\ln B_d(\eps))^{d-1}} = 
\frac{\eps\cdot A_d(\eps)}{(\ln A_d(\eps))^{d-1}}\cdot 
\left(\frac{\ln A_d(\eps)}{\ln B_d(\eps)}\right)^{d-1}\cdot
\frac{B_d(\eps)}{A_d(\eps)}\,.
$$
Together with \eqref{ws-19}, \eqref{Lem412}, and \eqref{f80} this implies
\be\label{ws-20_b}
  \lim_{\eps\downarrow 0}\frac{\eps\cdot B_d(\eps)}{(\ln B_d(\eps))^{d-1}}
= \lim_{\eps\downarrow 0}\frac{\eps\cdot A_d(\eps)}{(\ln A_d(\eps))^{d-1}}
=\frac{2^d}{(d-1)!}\,.
\ee


\subsubsection{The approximation numbers of $H^{s,+}_{\mix}(\T)$ in $L_2(\T)$}


In our first application of (\ref{ws-19}) we choose $b_{\ell} := (1+|\ell|^2)^{-1/2}$, $\ell\in \zz$. Then
\[
{\mathcal B}_d (\eps) := \Big\{k \in \Z: \: \prod_{j=1}^d b_{k_j} \ge \eps \Big\}
= \Big\{k \in \Z: \: 1/w_{1}^+(k) \ge \eps \Big\}\, , 
\]
for all $\varepsilon >0$\,.

\begin{cor}\label{lim2}
Let $d\in \N$. \\
{\rm (i)}
 Let $s>0$.  Then
 $$
 \lim_{n \to \infty} \, \frac{n^s\,a_{n} (I_d: \, H^{s,+}_\mix (\T) \to L_2 (\T))}{(\ln n)^{(d-1)s} \, } =
 \Big[\frac{2^{d}}{(d-1)!}\Big]^s \, .
 $$
{\rm (ii)}
 Let $s_0>s_1\ge 0$.  Then
 $$
 \lim_{n \to \infty} \, \frac{n^{s_0-s_1}\,a_{n} (I_d: \, H^{s_0,+}_\mix (\T) \to H^{s_1,+}_{\mix} (\T))}{(\ln n)^{(d-1){(s_0-s_1)}} \, } =
 \Big[\frac{2^{d}}{(d-1)!}\Big]^{s_0-s_1} \, .
 $$
\end{cor}

\bproof
It is enough to prove (i) for $s=1$. Indeed, then the known relation
$$
a_{n} (I_d: \, H^{s,+}_\mix (\T) \to L_2 (\T)) = a_{n} (I_d: \, H^{1,+}_\mix (\T) \to L_2 (\T))^s \,.
$$ 
implies (i) for arbitrary $s>0$, and the semigroup property of the weights yields (ii). 

\noindent
Setting $\varepsilon_r:=(1+r^2)^{-1/2}$ for $r\in\N_0$, we obviously have
$$
\{\varepsilon_r:\,r\in\N\}\subset\{1/w_1^+(k):\,k\in\Z\}\,,
$$
whence $a_n:=a_n(I_d: \, H^{1,+}_\mix (\T) \to L_2 (\T))=\varepsilon_r$\, if $n=B_d(\varepsilon_r)$\,. Consequently, if
$ B_d(\varepsilon_{r-1})\le n\le B_d(\varepsilon_r)$, then
$$
\varepsilon_r\le a_n\le\varepsilon_{r-1}\qquad\text{ and}\qquad 
\frac{\eps_r\cdot B_d(\eps_{r-1})}{(\ln B_d(\eps_{r-1}))^{d-1}}
\le\frac{n a_n}{(\ln n)^{d-1}}
\le\frac{\eps_{r-1}\cdot B_d(\eps_r)}{(\ln B_d(\eps_r))^{d-1}}\,.
$$
Since $\lim\limits_{r\to\infty} \varepsilon_{r-1}/\varepsilon_r =1$, a simple monotonicity argument and \eqref{ws-20_b} imply
$$
\lim\limits_{n\to\infty} \frac{n a_n}{(\ln n)^{d-1}}
=\lim_{\eps\downarrow 0}\frac{\eps\cdot B_d(\eps)}{(\ln B_d(\eps))^{d-1}}
= \lim_{\eps\downarrow 0}\frac{\eps\cdot A_d(\eps)}{(\ln A_d(\eps))^{d-1}}
=\frac{2^d}{(d-1)!}\,.
$$
\eproof

\noindent Corollary \ref{lim2} is the basis for the two-sided estimates of 
$a_{n} (I_d: \, H^{s,+}_\mix (\T) \to L_2 (\T))$
for large $n$ which we will study next.

\begin{satz}\label{gutt}
Let $s>0$ and  $d\in \N$. Then we have\\
 $$
  a_{n} (I_d: \, H^{s,+}_{\mix} (\T) \to L_2 (\T)) \le
\Big[\frac{(3\cdot \sqrt{2})^d}{(d-1)!}\Big]^s\frac{(\ln n)^{(d-1)s}}{n^s}\,, 
\qquad \mbox{if}\quad n \ge 27^d
 $$
and
$$
  a_{n} (I_d: \, H^{s,+}_{\mix} (\T) \to L_2 (\T)) \ge 
  \left[
  \frac{3}{d!}\, \Big(\frac{2}{2 + \ln 12}\Big)^{d}\right]^s\, \frac{(\ln n)^{(d-1)s}}{n^s} 
  \,,
\qquad 
\mbox{if}\quad n >  (12 \, e^2)^{d} \, .
 $$
\end{satz}

\bproof
By $I^j$, $j=1,2,3$, we denote identity mappings. \\
Since we have $\| \, I^1:H^{s,\#}_{\mix} (\T) \to H^{s,+}_{\mix} (\T)\| = 1$, the commutative diagram

\tikzset{node distance=5cm, auto}

\begin{center}
\begin{tikzpicture}
  \node (H) {$H^{s,\#}_{\mix}(\T)$};
  \node (L) [right of =H] {$H^{s,+}_{\mix}(\T)$};
  \node (L2) [right of =H, below of =H, node distance = 2.5cm ] {$L_2(\T)$};
    \draw[->] (H) to node {$I^1$} (L);
  \draw[->] (H) to node [swap] {$I^3$} (L2);
  \draw[->] (L) to node {$I^2$} (L2);
  \end{tikzpicture}
\end{center}
with $I^3= I^2 \circ I^1$ and basic properties of approximation numbers yields 
\[
a_{n} (I^3) \le \| \, I^1 \, |H^{s,\#}_{\mix} (\T) \to H^{s,+}_{\mix} (\T)\| 
\, a_{n} (I^2)\, .
\]
Supplemented by Theorem \ref{gut}/(ii) the lower estimate of  $a_{n} (I_d: \, H^{s,+} (\T) \to L_2 (\T))$ follows.
What concerns the upper bound we observe 
\[
\| \, I^1:H^{s,+}_{\mix} (\T) \to H^{s,\#}_{\mix} (\T)\| = 2^{ds/2}
\]
and consider the diagram
\begin{center}
\tikzset{node distance=5cm, auto}
\begin{tikzpicture}
  \node (H) {$H^{s,+}_{\mix}(\T)$};
  \node (L) [right of =H] {$H^{s,\#}_{\mix}(\T)$};
  \node (L2) [right of =H, below of =H, node distance = 2.5cm ] {$L_2(\T)$};
    \draw[->] (H) to node {$I^1$} (L);
  \draw[->] (H) to node [swap] {$I^3$} (L2);
  \draw[->] (L) to node {$I^2$} (L2);
  \end{tikzpicture}
\end{center}
with $I^3= I^2 \circ I^1$. This leads to 
\[
a_{n} (I^3) \le 2^{ds/2}\, \, a_{n} (I^2)\, .
\]
\eproof
\noindent Finally, we shall have a look at the behavior
of $a_{n} (I_d: \, H^{s,+}_\mix (\T) \to L_2 (\T))$
for small $n$. Recall that the quantity $\alpha (n,d)$ has been defined in \eqref{er-04}.\\

\begin{satz} \label{kleina}
Let $d\in \N$, $d\geq 2$, and $s>0$. Then for any $2\leq n \leq \frac{d}{2}4^d$ it holds the two-sided estimate
$$
    2^{-s}\Big(\frac{1}{n}\Big)^{\frac{s}{2+\alpha (n,d)}} \leq a_n(I_d:H^{s,+}_{\mix}(\T) \to L_2(\T)) \leq 
\Big(\frac{e^2}{n}\Big)^{\frac{s}{4+2\log_2 d}}\,.
$$
\end{satz}

\bproof The upper bound is a direct consequence of basic properties of approximation numbers, see Subsection \ref{parapp}, 
Lemma\ \ref{norm_one}/(v), Theorem\ \ref{small} and the 
first commutative diagram below, where $I^3= I^2 \circ I^1$.

\begin{minipage}[u]{7cm}
\tikzset{node distance=5cm, auto}
\begin{tikzpicture}
  \node (H) {$H^{s,+}_{\mix}(\T)$};
  \node (L) [right of =H] {$H^{s/2,\#}_{\mix}(\T)$};
  \node (L2) [right of =H, below of =H, node distance = 2.5cm ] {$L_2(\T)$};
    \draw[->] (H) to node {$I^1$} (L);
  \draw[->] (H) to node [swap] {$I^3$} (L2);
  \draw[->] (L) to node {$I^2$} (L2);
  \end{tikzpicture}
\end{minipage}
\begin{minipage}[u]{7cm}
 \tikzset{node distance=5cm, auto}
\begin{tikzpicture}
  \node (H) {$H^{s,\#}_{\mix}(\T)$};
  \node (L) [right of =H] {$H^{s,+}_{\mix}(\T)$};
  \node (L2) [right of =H, below of =H, node distance = 2.5cm ] {$L_2(\T)$};
    \draw[->] (H) to node {$I^4$} (L);
  \draw[->] (H) to node [swap] {$I^5$} (L2);
  \draw[->] (L) to node {$I^3$} (L2);
  \end{tikzpicture}
\end{minipage}\\
The lower bound follows from 
Lemma\ \ref{norm_one}/(i)-(iii), Theorem\ \ref{lowersmall} and 
the second commutative diagram above, where $I^5= I^3 \circ I^4$.
\eproof


\subsubsection{The approximation numbers of $H^{s,*}_{\mix}(\T)$ in $L_2(\T)$}


In our second application of (\ref{ws-19}) we choose $b_{\ell} := (1+|\ell|^{2s})^{-1/(2s)}$. This leads to 
\[
{\mathcal B}_d (\eps) := \Big\{k \in \Z: \: \prod_{j=1}^d b_{k_j} \ge \eps \Big\}
= \Big\{k \in \Z: \: 1/w_{s}^* (k) \ge \eps^{s} \Big\}\,.
\]
Due to the missing semigroup property we have to deal now with all $s>0$, not only with $s=1$.
But nevertheless we can proceed similarly as in the previous subsection. 

\begin{cor}\label{lim3}
Let $d\in \N$ and  $s>0$.  Then
 $$
 \lim_{n \to \infty} \, \frac{n^s\,a_{n} (I_d: \, H^{s,*}_\mix (\T) \to L_2 (\T))}{(\ln n)^{(d-1)s} \, } =
 \Big[\frac{2^{d}}{(d-1)!}\Big]^s \, .
 $$
\end{cor}
\bproof
Again \eqref{ws-20_b}
leads to 
$$
\lim\limits_{\eps\downarrow 0} \frac{\eps B_d(\eps)}{(\ln B_d(\eps))^{d-1}} = \frac{2^d}{(d-1)!}\,.
$$
Setting $\eps_r:=(1+ r^{2s})^{-1/2s}$, we have $\lim_{r\to\infty}\frac{\eps_r}{\eps_{r-1}}$\,, and 
$$
\{\eps_r^s: r\in\N_0\}\subset\{1/w_s^*(k): k\in\Z\}\,.
$$
Therefore, $\eps_r^s\le a_n\le \eps_{r-1}$\,, if $B_d(\eps_r)\le n \le B_d(\eps_{r-1})$. This gives
$$
\eps_r\le a_n^{1/s}\le \eps_{r-1}\qquad\mbox{and}\qquad
\frac{\eps_r\cdot B_d(\eps_{r-1})}{(\ln B_d(\eps_{r-1}))^{d-1}}
\le\frac{n a_n^{1/s}}{(\ln n)^{d-1}}
\le\frac{\eps_{r-1}\cdot B_d(\eps_r)}{(\ln B_d(\eps_r))^{d-1}}\,.
$$
Exactly the same arguments as in the proof of Theorem\ \ref{lim2} imply 
$$
\lim\limits_{n\to\infty} \frac{n a_n^{1/s}}{(\ln n)^{d-1}}
=\lim_{\eps\downarrow 0}\frac{\eps\cdot B_d(\eps)}{(\ln B_d(\eps))^{d-1}}
=\frac{2^d}{(d-1)!}\,,
$$
and this is equivalent to our assertion.
\eproof.

\noindent
Based on Corollary \ref{lim3} we can derive two-sided estimates of 
$a_{n} (I_d: \, H^{s,*}_\mix (\T) \to L_2 (\T))$
for large $n$.

\begin{satz}\label{guttt}
Let  $d\in \N$.\\
(i) Let $s>1/2$. Then
 $$
  a_{n} (I_d: \, H^{s,*}_{\mix} (\T) \to L_2 (\T)) \le 2^{-d/2}\, 
\Big[\frac{6^d}{(d-1)!}\Big]^s\,\frac{(\ln n)^{(d-1)s}}{n^s} \, 
\qquad \mbox{if}\quad n \ge 27^d \, ,
 $$
and 
 $$
  a_{n} (I_d: \, H^{s,*}_{\mix} (\T) \to L_2 (\T)) \ge \Big[
 \frac{3\cdot 2^d}{ d! \, (2+\ln 12)^d} \, \Big]^s \, \frac{(\ln n)^{(d-1)s}}{n^s}
\qquad 
\mbox{if}\quad n >  (12 \, e^{2})^d \, .
 $$
(ii)
Let $0 < s \le 1/2$. Then
 $$
  a_{n} (I_d: \, H^{s,*}_{\mix} (\T) \to L_2 (\T)) \le
\Big[\frac{3^d}{(d-1)!}\Big]^s\,  \frac{(\ln n)^{(d-1)s}}{n^s} \, \qquad 
\mbox{if}\quad n \ge 27^d \, ,
 $$
and 
 $$
  a_{n} (I_d: \, H^{s,*}_{\mix} (\T) \to L_2 (\T)) \ge 2^{-d/2}\, 
\Big[
\frac{3\cdot 4^d}{ d! \, (2+\ln 12)^d} \, \Big]^s \, \frac{(\ln n)^{(d-1)s}}{n^s} 
\qquad 
\mbox{if}\quad n >  (12 \, e^{2})^d \, .
 $$
\end{satz}

\bproof
We distinguish two cases: $s>1/2$ and $0 < s \le 1/2$.
\\
{\em Step 1.} Let $s>1/2$. Then
\[
\|I^1:H^{s,\#}_{\mix} (\T) \to H^{s,*}_{\mix} (\T)\| = 1 \qquad \mbox{and}
\qquad \|I^4:H^{s,*}_{\mix} (\T) \to H^{s,\#}_{\mix} (\T)\| = 2^{(s-1/2)d}\, .
\]
In view of the diagrams 

\begin{minipage}[u]{7cm}
\tikzset{node distance=5cm, auto}
\begin{tikzpicture}
  \node (H) {$H^{s,\#}_{\mix}(\T)$};
  \node (L) [right of =H] {$H^{s,*}_{\mix}(\T)$};
  \node (L2) [right of =H, below of =H, node distance = 2.5cm ] {$L_2(\T)$};
    \draw[->] (H) to node {$I^1$} (L);
  \draw[->] (H) to node [swap] {$I^3$} (L2);
  \draw[->] (L) to node {$I^2$} (L2);
  \end{tikzpicture}
\end{minipage} \hfill
\begin{minipage}[u]{7cm}
\tikzset{node distance=5cm, auto}
\begin{tikzpicture}
  \node (H) {$H^{s,*}_{\mix}(\T)$};
  \node (L) [right of =H] {$H^{s,\#}_{\mix}(\T)$};
  \node (L2) [right of =H, below of =H, node distance = 2.5cm ] {$L_2(\T)$};
    \draw[->] (H) to node {$I^4$} (L);
  \draw[->] (H) to node [swap] {$I^2$} (L2);
  \draw[->] (L) to node {$I^3$} (L2);
\end{tikzpicture}
\end{minipage}\\
with $I^3= I^2 \circ I^1$, $I^2= I^3 \circ I^4$, this yields
\[
a_{n} (I^3) \le \, a_{n} (I^2) \le 2^{(s-1/2)d}\, a_n (I^3)\, .
\]
Now the claimed estimates follow from  Theorem \ref{gut}.
\\
{\em Step 2.} Let $0 < s \le 1/2$. Then
\[
\|I^1:H^{s,\#}_{\mix} (\T) \to H^{s,*}_{\mix} (\T)\| = 2^{(\frac 12 - s )d} \qquad \mbox{and}
\qquad \|I^4:H^{s,*}_{\mix} (\T) \to H^{s,\#}_{\mix} (\T)\| = 1\, .
\]
Employing the same diagrams as in Step 1 we conclude
\[
a_{n} (I^2) \le \, a_{n} (I^3) \le 2^{(\frac 12-s)d}\, a_n (I^2)\, .
\]
Now the claimed estimates follow from  Theorem \ref{gut}.
\eproof

\noindent Again, as the last step in this subsection, we shall consider the behavior
of $a_{n} (I_d: \, H^{s,*}_\mix (\T) \to L_2 (\T))$
for small $n$. This time we have only partial results.
\\

\begin{satz}\label{kleinb} 
Let $d\in \N$, $d\geq 2$, and $1/2 \le s \le 1$. Then for any $2\leq n \leq \frac{d}{2}4^d$ it holds the two-sided estimate
$$
    2^{-s}\Big(\frac{1}{n}\Big)^{\frac{s}{2+\alpha (n,d)}} \leq a_n(I_d:H^{s,*}_{\mix}(\T) \to L_2(\T)) \leq 
\Big(\frac{e^2}{n}\Big)^{\frac{s}{4+\log_2 (d^2)}}\,.
$$
\end{satz}

\bproof 
We argue as in the proof of Theorem\ \ref{kleina}.
The upper bound is a consequence of Lemma\ \ref{norm_one},(ii) together with Theorem\ \ref{kleina} and the first commutative diagram

\begin{minipage}[u]{7cm}
\tikzset{node distance=5cm, auto}
\begin{tikzpicture}
  \node (H) {$H^{s,*}_{\mix}(\T)$};
  \node (L) [right of =H] {$H^{s,+}_{\mix}(\T)$};
  \node (L2) [right of =H, below of =H, node distance = 2.5cm ] {$L_2(\T)$};
    \draw[->] (H) to node {$I^1$} (L);
  \draw[->] (H) to node [swap] {$I^3$} (L2);
  \draw[->] (L) to node {$I^2$} (L2);
  \end{tikzpicture}
\end{minipage} \hfill
\begin{minipage}[u]{7cm}
\tikzset{node distance=5cm, auto}
\begin{tikzpicture}
  \node (H) {$H^{s,\#}_{\mix}(\T)$};
  \node (L) [right of =H] {$H^{s,*}_{\mix}(\T)$};
  \node (L2) [right of =H, below of =H, node distance = 2.5cm ] {$L_2(\T)$};
    \draw[->] (H) to node {$I^4$} (L);
  \draw[->] (H) to node [swap] {$I^5$} (L2);
  \draw[->] (L) to node {$I^3$} (L2);
\end{tikzpicture}
\end{minipage}\\
with $I^3= I^2 \circ I^1$. 
The lower bound follows from
Lemma\ \ref{norm_one}, (ii), Theorem\ \ref{lowersmall} and 
the second commutative diagram using $I^5 = I^3 \circ I^4$.
\eproof


\subsection{The approximation numbers of $H^{m}_{\mix}(\T)$ in $L_2 (\T)$}

By setting $b_{\ell} = v_m(\ell)^{1/m}$, see \eqref{f61}, we could argue similar as in the previous subsection 
to compute $\lim_{n\to \infty} n^m\cdot a_n/(\ln n)^{(d-1)m}$.
However, Lemma \ref{norm_two} provides a much simpler argument.

\begin{cor}\label{lim4}
Let $d\in \N$ and  $m \in \N$.  Then 
 $$
 \lim_{n \to \infty} \, \frac{n^m\,a_{n} (I_d: \, H^{m}_\mix (\T) \to L_2 (\T))}{(\ln n)^{(d-1)m} \, } =
 \Big[\frac{2^{d}}{(d-1)!}\Big]^m \, .
 $$
\end{cor}

\bproof
This follows immediately from Corollaries \ref{lim2} and \ref{lim3}, and 
\beq\label{ws-22}
a_{n} (I_d: \, H^{m,+}_\mix (\T) \to L_2 (\T)) & \le &   a_{n} (I_d: \, H^{m}_\mix (\T) \to L_2 (\T)) 
\nonumber
\\
& \le &  a_{n} (I_d: \, H^{m,*}_\mix (\T) \to L_2 (\T))
\eeq
which is itself a consequence of Lemma  \ref{norm_two}.
\eproof

\noindent
Based on Thms.\ \ref{gutt}, \ref{guttt} and (\ref{ws-22}) we derive  two-sided estimates of 
$a_{n} (I_d: \, H^{m}_\mix (\T) \to L_2 (\T))$
for large $n$.

\begin{satz}\label{gutttt}
Let $m \in \N$ and  $d\in \N$.\\
(i) Then
 $$
  a_{n} (I_d: \, H^{m}_{\mix} (\T) \to L_2 (\T)) \le
 \Big[\frac{6^d}{(d-1)!}\Big]^m \, 2^{-d/2}\, 
 \frac{(\ln n)^{(d-1)m}}{n^m}\, \qquad \mbox{if}\quad n \ge 27^d \, .
 $$
 (ii) In addition
 $$
  a_{n} (I_d: \, H^{m}_{\mix} (\T) \to L_2 (\T)) \ge 
\Big[\frac{3\cdot 2^d}{ d! \, (2+\ln 12)^d} \, \Big]^s
 \frac{(\ln n)^{(d-1)m}}{n^m}
\qquad 
\mbox{if}\quad n > (12 \, e^{2})^d \, .
 $$
\end{satz}

\begin{rem}
 \rm
Also for the embedding $H^{m}_{\mix} (\T) \to L_2 (\T))$ the behavior of 
$a_{n} (I_d: \, H^{m}_{\mix} (\T) \to L_2 (\T))$ for small $n$ is of interest.
By the coincidence $H^1_{\mix}(\T) = H^{1,*}_{\mix}(\T)$ (equal norms) we obtain the relations
$$
    \frac 12 \, \Big(\frac{1}{n}\Big)^{\frac{1}{2+\alpha (n,d)}} \leq a_n(I_d:H^{1}_{\mix}(\T) \to L_2(\T)) \leq 
\Big(\frac{e^2}{n}\Big)^{\frac{1}{4+\log_2 (d^2)}}\,
$$
immediately from Theorem\ \ref{kleinb}.
\end{rem}


\subsection{Various comments on the literature}\label{lit}


\begin{itemize}
 \item Closest to us in aims and methods is the recent paper  \cite{DiUl13}. There, in Theorem\ 3.13, 
the authors obtained for $s>0$ and any $n  \ge 2^d$ the inequality
\begin{equation} \label{zungtino} 
a_n(I_d:H^{s, \Box}_{\mix}(\T)\to L_2(\T))
\ \le \
4^s \Big(\frac{2e}{d-1}\Big)^{s(d-1)} n^{-s} (\log_2 n)^{s(d-1)},
\end{equation}
where $\Box$ indicates that the space $H^{s, \Box}_{\mix}(\T)$ is equipped with a further norm
(based on dyadic decompositions on the Fourier side and different from those studied here).
This has to be compared with 
Theorems \ \ref{gut}, \ref{gutt}, \ref{guttt} and \ref{gutttt}.
In all these cases we have a super-exponential decay of the constants $C_s (d)$ in $d$.
\item
Super-exponential decay of the constants $C_s (d)$ in $d$  
has been observed before. Bungartz and Griebel \cite[Theorem\ 3.8]{BG04} investigated the non-periodic situation.
An approximation is given with respect to tensor products of piecewise linear functions. 
The authors proved that for any  $n \in \N$ there exists a 
subspace $V_n^{(1)} \subset L_2 ([0,1]^d)$ with $m = m(d,n)$ degrees of freedom 
and a projection $Q_n$ onto $V_n^{(1)}$ such that 
\be\label{griebel}
 \|\, f- Q_n f\, |L_2([0,1]^d)\| \leq \frac{2}{12^d}2^{-2n} A(d,n)
    \, \Big\|\, \frac{\partial^{2d}f}{\partial x_1^2 \, \ldots \, \partial x_d^2}\, 
    \Big|L_{2}([0,1]^d)\Big\|
\ee
holds for all continuous functions $f$ vanishing on the boundary $\partial([0,1]^d)$. The latter assumption 
is actually crucial for the explicit bound in \eqref{griebel}.
Here, the number $A(d,n)$ is given by 
$$
  A(d,n) := \sum_{k = 0}^{d-1}\binom{n+d-1}{k}\,.$$ 
Of course, inequality \eqref{griebel} does not allow for a comparison 
with the results obtained in this paper. One first has to rewrite the bounds in terms of the degrees of freedom $m$. 
For this issue a careful 
two-sided estimate of $m = \dim V^{(1)}_n$ is required first.
Lemma 3.6  in \cite{BG04} shows that $\dim V^{(1)}_n$ can be estimated from 
above and below by 
\be\label{dimVn}
2^{n-1} \, \binom{n+d-2}{d-1} \le   \dim V^{(1)}_n = \sum_{j=0}^{n-1} 2^j \binom{j+d-1}{d-1} \le  2^{n}\,  \binom{n+d-2}{d-1}\, .
\ee
For $n\ge d-1$ the expression $\binom{n+d-1}{k}$ is increasing in $k\le d-1$. In this case we have the estimate
\[
A(d,n) \leq d\binom{n+d-1}{d-1}\,.
\]
We will now transfer \eqref{griebel} to the notion of approximation numbers. To be precise we consider the space/norm
\be\nonumber
  \begin{split}
    H^2_{\mix,0}([0,1]^d):= \{ f \in L_2([0,1]^d)~:~&\|\, f\, |H^{2}_{\mix}([0,1]^d)\|:=
\max_{\alpha \in \{0,2\}^d} \|\, D^\alpha f\, |L_{2}([0,1]^d)\|<\infty\\ 
&\mbox{ and } f=0 \quad \mbox{on the boundary}\}\,.
  \end{split}
\ee
Note, that $\|\, f\, |H^{2}_{\mix}([0,1]^d)\|$ in the previous formula represents a very weak norm in $H^2_{\mix,0}([0,1]^d)$ compared to \eqref{t-1}. 
Based on \eqref{griebel}, \eqref{dimVn} and the monotonicity of approximation numbers 
we find for any $m$ satisfying $\dim V_{n}^{(1)}<m\leq \dim V_{n+1}^{(1)}$ for some $n\geq d-1$ the relation
\[
\begin{split}
 a_m(I_d:H^2_{\mix,0}([0,1]^d) \to L_2([0,1]^d)) &\leq \frac{2}{12^d}2^{-2n} A(d,n)  \leq \frac{8d}{12^d} \, m^{-2} \, 
 \binom{n+d-1}{d-1}^{3}\\ &\leq \frac{8d}{12^d} \, m^{-2} \, 
 \binom{\log_2 m +d}{d-1}^{3}\, \\
&\leq \frac{8d}{12^d} \, m^{-2} \, 
 \Big(\frac{2\, e \, \log_2 m}{d-1}\Big)^{3(d-1)}\, .
\end{split}
\]
Consequently, inequality \eqref{griebel} implies for any $m\geq 2^{d-1}\binom{2d-2}{d-1}$ that
\be\label{gr2}
a_m(I_d:H^2_{\mix,0}([0,1]^d) \to L_2([0,1]^d))\leq C(d)\, m^{-2}(\log_2 m)^{3(d-1)}\,,
\ee
where 
\be\nonumber
    C(d) := \frac{2d}{3}\, \Big(\frac{2e}{12^{1/3}(d-1)}\Big)^{3(d-1)}\,.
\ee
This constant $C(d)$ is decaying extremely fast, i.e., super-exponentially in $d$, 
similar as in Thms. \ref{gut}, \ref{gutt}, \ref{guttt}, \ref{gutttt} or in \eqref{zungtino} above.  
But comparing \eqref{gr2} with  the  quoted estimates of the approximation numbers 
$a_n (I_d:H^2_{\mix}(\tor^d), L_2(\tor^d))$, then it is obvious 
that the power of the logarithm $3(d-1)$ in \eqref{gr2}
is larger than there, where it is always  $2(d-1)$. This is at least partly caused by 
the fact that interpolation operators of Smolyak type 
are known to be not optimal in the sense of approximation numbers in such a context and Bungartz and Griebel 
are using  an interpolation operator of Smolyak type with respect to a sparse grid. 
However, a reasonable comparison of \eqref{griebel} and  Thms. \ref{gut}, \ref{gutt}, \ref{guttt}, \ref{gutttt} or  \eqref{zungtino}
can not be made because of the following reasons:
\begin{itemize}
\item[(i)] The periodic Sobolev spaces $H^2_{\mix}(\tor^d)$ are smaller than the non-periodic 
Sobolev spaces $H^2_{\mix}([0,1]^d)$ (and the ``difference'' is increasing with $d$).
\item[(ii)] The space $H^2_{\mix,0}([0,1]^d)$
is much smaller than the original space $H^2_{\mix}([0,1]^d)$.
\item[(iii)] On the right-hand side in \eqref{griebel} only the term 
$\displaystyle \|\, \frac{\partial^{2d}f}{\partial x_1^2 \, \ldots \, \partial x_d^2}\, |L_{2}([0,1]^d)\|$
shows up, which is much smaller than the full norm used in our investigations above.
\end{itemize}

\item Preasymptotics. The inequalities \eqref{griebel}
and \eqref{dimVn} remain also true for small $n$. Note, that in case $1  \le n\leq d-1$ the number $A(d,n)$ can be estimated  as follows 
$$
\frac 12 \, 2^{n+d-1}\le    A(d,n) = \sum_{k = 0}^{d-1}\binom{n+d-1}{k} \le 2^{n+d-1} 
$$
(we sum up to $d-1$, which is larger than $(n+d-1)/2$). Let $\dim V_n^{(1)} \leq m \leq \dim V_{n+1}^{(1)}$ for some $1 \leq n \leq d-1$. Using 
\eqref{dimVn}, \eqref{griebel} and the space $H^2_{\mix,0}([0,1]^d)$ defined above we obtain
\be\nonumber
\begin{split}
a_m(I_d:H^2_{\mix,0}([0,1]^d) \to L_2([0,1]^d)) &\leq \frac{2}{12^d}2^{-2n} A(d,n)\\ 
&\le \frac{4}{12^d}m^{-1}2^{n+d-1}2^{-n}\binom{n+d-1}{d-1}\\
&\leq \frac{1}{3}\Big(\frac{4e}{12}\Big)^{d-1}m^{-1} = \frac{1}{e}\Big(\frac{e}{3}\Big)^dm^{-1}\,.
\end{split}
\ee
Hence we obtain for $1\leq m\leq 2^{d-1}\binom{2d-2}{d-1}$ the ``preasymptotic'' decay
\be\label{t-5}
   a_m(I_d:H^2_{\mix,0}([0,1]^d) \to L_2([0,1]^d)) \leq C(d)m^{-1}
\ee
with 
\[
C(d):= \frac 1 e \, \Big(\frac{e}{3}\Big)^d \, .
\]
This time we ``just have'' exponential decay of the constant $C(d)$ in $d$.
Now we  compare this with our results obtained in 
Thms. \ref{small}, \ref{kleina} and  \ref{kleinb}.
Let us concentrate on \eqref{ws-46}.
There we proved for a range in $1\leq m \leq 4^d$ the inequality
\[
a_m(I_d:H^{2,\#}_{\mix}(\T) \to L_2(\T)) \leq \Big(\frac{e^2}{m}\Big)^{\frac{1}{1+ \log_2 \sqrt{d}}}\,.
\]
Inequality \eqref{t-5} looks much better than this inequality (with respect to the exponent of $m$ as well as with respect to the $d$-dependence of the constant).
This is mainly caused by the homogeneous boundary condition in \eqref{t-5} which shrinks the space significantly. 

\item
Also Schwab, S\"uli, and  Todor \cite{SST08} considered the non-periodic situation. A particular case of 
their  results in \cite[Theorem\ 5.1]{SST08} can be formulated as follows. 
For $s\in \N$, $s\geq 2$ and $L\geq d-1$ there exists a 
subspace $V_L$  of $L_2 ([0,1]^d)$ with $m=\dim V_L$ degrees of freedom  and a projection $Q_L$ onto $V_L$ such that 
\be\label{schwab}
\|\, f - Q_L f\, |L_2([0,1]^d)\| \leq C_2(d,s) m^{-s}(\log m)^{(d-1)(s+1)}\, 
\|\, f\, |H^{s, \triangle}_{\mix}([0,1]^d)\|\,,
\ee
where $C_2(d,s)$ decays super-exponentially in $d$. 
Here the authors use the stronger norm
$$
\|\, f\, |H^{s, \triangle}_{\mix}([0,1]^d)\|:=
\sum_{\substack{0\leq \alpha_i \leq s\\ i=1,...,d}} \|\, D^\alpha f\, |L_{2}([0,1]^d)\|\,,
$$
where $s\in \N$. 
The result is stated in \cite{SST08} in a slightly different form.  
As above one has to supplement an inequality similar to \eqref{griebel} by 
two-sided estimates for $\dim V_L$ to turn it into \eqref{schwab}.
In some sense \eqref{schwab} generalizes \eqref{griebel} to the case of higher smoothness.
As before the power of the logarithm is worse compared 
with the behavior of the approximation numbers
$a_m (I_d:H^s_{\mix}(\tor^d), L_2(\tor^d))$.

\item
Preasymptotics. In case $L\leq d-1$ Schwab, S\"uli, and  Todor \cite{SST08} 
obtained, under additional restrictions, the  estimate 
$$
\|\, f - Q_L f\, |L_2([0,1]^d)\| \leq C_3(d,s)  2^{-Ls}\, 
\|\, f\, |H^{s,\triangle}_{\mix}([0,1]^d)\|\,,
$$
where $C_3(d,s)$ decays exponentially in $d$. 
Again this has to be complemented by a two-sided estimate of $m:= \dim V_L$.
A rather rough but sophisticated estimate
yields
\[
\|\, f - Q_L f\, |L_2([0,1]^d)\| \leq C_4(d,s)  m^{-{s}/{3}}\, 
\|\, f\, |H^{s,\triangle}_{\mix}([0,1]^d)\|\,, \qquad m \le 2^{2(d-1)}\, , 
\]
where we have been unable to clarify the dependence of the constant $C_4(d,s)$ on $d$.

\item
Neither Bungartz and Griebel \cite{BG04} nor Schwab, S\"uli, and  Todor \cite{SST08} considered 
estimates from below.

\item
Sampling operators versus general linear operators.
As mentioned above the estimates \eqref{gr2} and \eqref{schwab}
are obtained by using  interpolation operators $Q_n$ with $m(d,n)$ sample points 
based upon univariate spline interpolation operators (via a Smolyak construction). 
Let us mention the following result in this context:
in \cite[Theorem\ 6.2,(i)]{SU10} we show for $1/2<s<2$ and all $m \in \N$
\be\label{er-07}
     \|\, f- A_m f\, |L_2([0,1]^d)\| \lesssim m^{-s}(\log m)^{(d-1)(s+1/2)}\, \|\, f \, |H^s_{\mix}([0,1]^d)\|\,.
\ee
Compared to the estimates \eqref{gr2} and \eqref{schwab} we improved the power of the logarithm by $(d-1)/2$, however,
we do not know about the $d$-dependence of the constants in \eqref{er-07}.
The restrictions on $s$ in \eqref{er-07} are caused by the fact that we worked with piecewise linear functions. 
In D\~ung \cite{DD10} the relation \eqref{er-07} has been  extended to all $s>1/2$ via $B$-spline quasi-interpolation
(but also without taking care of the $d$-dependence of the constants).

\item
Motivated by the aim to approximate the solution of a Poisson equation in the energy norm, i.e., in the norm of the isotropic 
Sobolev space $H^1$, Bungartz and Griebel \cite{BG99} investigated estimates of the quantities $a_n (I_d:H^2_{\mix}([0,1]^d), H^1([0,1]^d))$.
These studies have been continued in Griebel, Knapek \cite{GN00, GN09}, Bungartz, Griebel \cite{BG04}, Griebel \cite{Gr}, 
Schwab, S\"uli, and  Todor \cite{SST08}, and D\~ung, Ullrich \cite{DiUl13}.
Let us comment on the non-periodic situation first. It was already noticed by Griebel in \cite[Theorem\ 2]{Gr} that in this situation
the constant (in front of the approximation order term) decays exponentially in $d$. To be more precise, 
he proved that there is a subspace $V_n$ with $n$ degrees of freedom and a projection $Q_n$ onto $V_n$ such that for large $n$
\be\label{f60}
 \|\, f- Q_n f\, |H^1 ([0,1]^d)\| \leq c\cdot c_1(d)\cdot c_2(d)\cdot  n^{-1} \, 
\Big\|\, \frac{\partial^{2d}f}{\partial x_1^2...\partial x_d^2 }\, \Big|L_{\infty}[0,1]^d\Big\|\,, 
\ee
holds, where
$$
c_1(d) = \frac{d}{2}e^d\quad \mbox{ and }\quad  c_2(d) = \frac{d}{3^{(d-1)/2} \, 4^{d-1}}
\Big[\frac 12 + \Big(\frac{5}{2}\Big)^{d-1}\Big]\,.$$
Hence, the product $c_1(d)c_2(d)$ decays like $d^2\cdot 0.980875^d$\,.
Note, that the $L_{\infty}$-norm is involved in \eqref{f60} and the functions $f$ are taken from spaces 
with mixed smoothness of order $2$ and homogeneous boundary conditions.  
The situation changes significantly if one replaces $L_{\infty}$ by $L_2$ in \eqref{f60}. The source space for $f$ is now getting larger and hence
the approximation is getting worse. In \cite[Table 3.2, page 35]{BG04} the constant $c_2(d)$ can be chosen as 
$$
  c_2(d)  = \frac{2d}{\sqrt{3}\cdot 6^{d-1}}
\Big[\frac 12 + \Big(\frac{5}{2}\Big)^{d-1}\Big]\,.
$$
Therefore, $c_1(d)c_2(d)$ can only be estimated by $d^2\cdot 1.1326^d$ and thus an exponential decay can not longer be guaranteed. However, 
if the smoothness $s$ of the source space is less than $2$ we can say something in the periodic setting. From  
\cite[Theorem\ 3.6,(ii)]{DiUl13} it follows that if the error is measured in $H^1(\T)$ and $s<2$ we get for $n>\lambda^d$ (for some
$\lambda > 1$) the relation
$$
  a_n(I_d:\tilde{H}^s_{\mix}(\T) \to H^1(\T)) \leq cd^{s-1}\Big(\frac{1}{2^{1/(s-1)}-1}\Big)^d \, n^{-(s-1)}\,.
$$
Here $\tilde{H}^s_{\mix}(\T)$ is the subspace of $H^s_{\mix}(\T)$ containing all functions $f$ such that 
$c_k(f) \neq 0 \implies \prod_{i=1}^d k_i \neq 0$\,. This technical condition is essential to prove \eqref{f60}, see \cite[Theorem\ 2]{Gr}. 
Without this condition, i.e., for the 
entire space $H^s_{\mix}(\T)$ normed with dyadic building blocks on the Fourier side, we can disprove the exponential decay 
of the constants if $s\leq 2$, see \cite[Theorem\ 4.7,(i)]{DiUl13}. 
\end{itemize}


\section{Quasi-polynomial tractability}\label{Erich}


Now we will translate our results to recent tractability notions. Various concepts
of tractability are discussed in the recent monographs by Novak and Wo{\'z}niakowski 
\cite{NoWo08, NoWo10, NoWo12}\,. We will obtain ``quasi-polynomial tractability'' of the respective approximation
problems. This notion has been recently
introduced in \cite{GnWo11} and is a stronger notion than ``weak tractability''\,.


\subsection{General notions of tractability}


For arbitrary $s>0$ and all $d\in \N$ we consider the embedding operators (formal identities)
$$
    I_d:H^s_{\mix}(\T) \to L_2(\T)\,,
$$
where the Sobolev spaces are equipped with the norms $\|\cdot|H^s_{\mix}(\T)\|$, $\|\cdot|H^s_{\mix}(\T)\|^*$,  and
$\|\cdot|H^s(\T)\|^\#$\,. In both cases we have $\|I_d\| = 1$ for all $s>0$ and $d\in \N$. In other
words, the {\em normalized error criterion} is satisfied. In this context, a {\em linear algorithm} that uses
{\em arbitrary information} is of the form
$$
  A_{n,d}(f) = \sum\limits_{j=1}^n L_j(f) g_j\,,
$$
where $g_j \in L_2(\T)$ and $L_j$ are continuous linear functionals. If the error is measured in
the norm of $L_2(\T)$ we can identify the algorithm $A_{n,d}$ with a bounded linear
operator $A_{n,d}:H^s_{\mix}(\T) \to L_2(\T)$ of rank at most $n$. The worst-case error of $A_{n,d}$
with respect to the unit ball (respective norms) in $H^s_\mix(\T)$
$$
    \sup\limits_{\|f|H^s_{\mix}(\T)\|\leq 1} \|f-A_{n,d}(f)|L_2(\T)\|
$$
clearly coincides with the operator norm $\|I_d-A_{n,d}: H^s_{\mix}(\T) \to L_2(\T)\|$,
and the {\em $n$th minimal worst-case error} with respect to {\em linear algorithms
and general information}
$$
    \inf\limits_{\rank A_{n,d} \leq n} \|I-A_{n,d}:H^s_{\mix}(\T)\to L(\T)\|
$$
is just the {\em approximation number} $a_{n+1}(I_d:H^s_{\mix}(\T) \to L_2(\T))$, see \eqref{0002}\,.

Finally, the {\em information complexity} of the $d$-variate approximation problem
is measured by the quantity $n(\varepsilon, d)$ defined by
$$
   n(\varepsilon,d):=\inf\{n\in \N:a_n(I_d) \leq \varepsilon\}
$$
as $\varepsilon\to 0$ and $d\to \infty$\,. The approximation problem is called {\em weakly tractable}, if
\be\label{eq55}
    \lim\limits_{1/\varepsilon+d\to\infty} \frac{\ln n(\varepsilon,d)}{1/\varepsilon + d} = 0\,,
\ee
i.e., $n(\varepsilon,d)$ neither depends exponentially on $1/\varepsilon$ nor on $d$. The problem
is called {\em intractable}, if \eqref{eq55} does not hold, see the definition \cite[p.\ 7]{NoWo08}.
If for some $0<\varepsilon<1$ the number $n(\varepsilon,d)$ is an exponential function in $d$ then
we say that the approximation problem suffers from {\em the curse of dimensionality}. In other words, if there
exist positive numbers $c, \varepsilon_0, \gamma$ such that
$$
        n(\varepsilon,d) \geq c(1+\gamma)^d\,,\quad \mbox{for all } 0<\varepsilon\leq \varepsilon_0 \mbox{ and infinitely many }d\in \N\,,
$$
then the problem suffers from
{\em the curse of dimensionality}\,.\\
We need a further notion of tractability, namely {\it quasi-polynomial tractability}, see for instance \cite{GnWo11}\,. In fact,
the approximation problem is called {quasi-polynomially tractable} if there are positive numbers $t$ and $C_t$ such that 
\be\label{eq32}
        n(\varepsilon,d) \leq C_t\exp(t\ln(\varepsilon^{-1})(1+\ln(d)))\,.
\ee

Of course, quasi-polynomial tractability implies weak tractability.


\subsection{Tractability results for $H^{s}_{\mix}(\T)$}


By our results in Section \ref{number} we are very well prepared for 
the investigation of these tractability problems,
resulting in short proofs of the assertions. 

\begin{satz}\label{qpt1} For every $s>0$ the approximation problem for the embeddings
$$
    I_d:H^{s,\#}_\mix(\T) \to L_2(\T)
$$
is quasi-polynomially tractable. 
\end{satz}

\bproof 
Let $0<\varepsilon\leq 1$ be given. Select $r\in \N$ such that $r^{-s} < \varepsilon \leq (r-1)^{-s}$. 
Employing Lemma\ \ref{an1} we get $a_{C(r,d)}(I_d:H^{s,\#}_{\mix}(\T)\to L_2(\T)) = 1/r^s < \varepsilon$\,. This implies 
$$
 n(\varepsilon,d) \leq n(r^{-s},d) \leq C(r,d)\,.
$$
Once more, the problem reduces to the estimation of $C(r,d)$. Parts of it have already been done in the proof of Theorem\ \ref{small} above. There
we observed that $C(r,d) \leq e^2r^{2+\log_2 d}$ if $r\leq 2^d$. In case $r\ge  e^d$ we proved 
$C(r,d) \leq r^2 \, 3^d \leq r^{4}$, see (\ref{ws-25}). It remains to deal with $2^d < r < e^d$.
Based on Lemma \ref{ook}, Lemma \ref{vol} and taking into account  $x^k/k! \le e^x$ applied to $x= \ln r$ we conclude
\[
    C(r,d)  \le   1+  \sum\limits_{\ell = 1}^{d} \binom{d}{\ell}\, 2^\ell \, r\,  \frac{(\ln r)^{\ell-1}}{(\ell-1)!} 
 \leq   r^2 \, \sum\limits_{\ell = 1}^{d} \binom{d}{\ell}\, 2^\ell 
 \le  3^d \,   r^2 \le r^4\, .
\]
This gives
\beq\label{ws-26}
    \ln(n(\varepsilon,d)) \leq \ln C(r,d) \leq \left\{\begin{array}{rcl}  2 + (2 + \log_2 d)\ln r &:& r\leq 2^d\, ,
\nonumber
\\
  4\,  \ln r&:& r\ge 2^d\,.
    \end{array}\right.
\eeq
Since $\ln r \sim \ln(1/\varepsilon)/s$ we obtain \eqref{eq32}\,. The proof is complete. 
\eproof

\begin{cor} \label{qpt2} 
For every $s>0$ and every $m\in\N$ the approximation problems for the embeddings
$$
    I_d:H^{s,+}_\mix(\T) \to L_2(\T)\quad\mbox{,}\quad I_d:H^{s,\ast}_\mix(\T) \to L_2(\T)\quad\mbox{and}\quad I_d:H^{m}_\mix(\T) \to L_2(\T)
$$
are quasi-polynomially tractable. 
\end{cor}

\bproof The results are direct consequences of the embeddings in Lemma\ \ref{norm_one}. Because of Lemma\ \ref{norm_one}/(iv) it is sufficient to consider $s<1/2$. Then Lemma\ \ref{norm_one}/(iii) and Theorem\ \ref{qpt1} immediately imply the statement for $H^{s,*}_{\mix}(\T)$. Finally, Lemma\ \ref{norm_one}/(v) and Theorem\ \ref{qpt1} give quasi-polynomial tractability for $H^{s,+}_{\mix}(\T)$.
\eproof

\begin{rem} \rm Tensor product problems play an essential role in IBC (information based complexity), see, 
e.g., Section  2 in Chapter 5 and Section  2 in Chapter 8 of the monograph 
\cite{NoWo08}. The spaces $H^m_{\mix}(\T)$, $H^{s,\#}_\mix(\T)$, $H^{s,+}_\mix(\T)$, and $H^{s,*}_\mix(\T)$ are $d$-fold tensor products of the univariate spaces $H^m(\tor)$, $H^{s,\#}(\tor)$, $H^{s,+}(\tor)$, and $H^{s,*}(\tor)$, respectively,
see for instance \cite{SU09}. Obviously, the identity $I_d$ is a compact tensor product operator (considered as a mapping into the 
tensor product space $L_2 (\T)$).
Since the approximation numbers decay polynomially in these four univariate situations and $1 = a_1 > a_2$\,, we obtain the following conclusion from the general Theorem\ 3.3 of \cite{GnWo11}: For any $s>0$, all four problems $I_d:H^{m}_\mix(\T) \to L_2(\T)$, $I_d:H^{s,\#}_\mix(\T) \to L_2(\T)$, $I_d:H^{s,+}_\mix(\T) \to L_2(\T)$, and $I_d:H^{s,*}_\mix(\T) \to L_2(\T)$
are quasi-polynomially tractable (and polynomially intractable). Let us also mention, that this result in \cite{GnWo11}
has a forerunner in \cite{NoWo09}, where it has been proven that such tensor product problems are weakly tractable. Hence, 
Theorem\ \ref{qpt1} and Corollary\ \ref{qpt2} are special cases of a more general theory.
However, the approach given here is different.
\end{rem}

{\bf Acknowledgments.}
The authors would like to thank the organizers of the Dagstuhl seminar 12391 ``Algorithms and Complexity for Continuous Problems'', 2012,
where this work has been initiated, for providing a pleasant and fruitful working atmosphere. In addition, the authors
would like to thank Jan Vyb{\'i}ral and Glenn Byrenheid for commenting on earlier versions of this manuscript.

\end{document}